\documentclass[a4paper, 12pt]{amsart}
\usepackage[all]{xy}
\usepackage{hyperref}
\usepackage{tikz}

\numberwithin{equation}{section}

\theoremstyle{theorem}
\newtheorem{thm}{Theorem}[section]
\newtheorem{lemma}[thm]{Lemma}

\newtheorem{cor}[thm]{Corollary}

\theoremstyle{definition}
\newtheorem{defn}[thm]{Definition}

\theoremstyle{remark}
\newtheorem{eg}[thm]{Example}
\newtheorem{rmk}[thm]{Remark}

\newcommand{\id}{\operatorname{id}}
\newcommand{\lsp}{\operatorname{span}}
\newcommand{\clsp}{\overline{\lsp}}
\newcommand{\FE}{\operatorname{FE}}
\newcommand{\Tr}{\operatorname{Tr}}

\newcommand{\CC}{\mathbb{C}}
\newcommand{\NN}{\mathbb{N}}
\newcommand{\TT}{\mathbb{T}}
\newcommand{\ZZ}{\mathbb{Z}}
\newcommand{\RR}{\mathbb{R}}
\newcommand{\QQ}{\mathbb{Q}}

\newcommand{\minimatrix}[4]{\bigl(\begin{smallmatrix}#1 & #2 \\ #3 & #4\end{smallmatrix} \bigr)}
\newcommand{\minivector}[2]{\bigl(\begin{smallmatrix}#1 \\ #2 \end{smallmatrix} \bigr)}

\newcommand{\dstar}{d_*}

\title[Twisted higher-rank-graph $C^*$-algebras associated to Bratteli diagrams]
{Twisted $k$-graph algebras associated to Bratteli diagrams}

\date{18 March 2014}

\author{David Pask}
\email{dpask, asierakow, asims@uow.edu.au}
\author{Adam Sierakowski}
\author{Aidan Sims}
\address{School of Mathematics and Applied Statistics \\
University of Wollongong\\
Wollongong NSW 2522\\
AUSTRALIA}

\thanks{This research was supported by the Australian Research Council.}

\begin{document}

\begin{abstract}
Given a system of coverings of $k$-graphs, we show that the cohomology of the resulting
$(k+1)$-graph is isomorphic to that of any one of the $k$-graphs in the system. We then
consider Bratteli diagrams of 2-graphs whose twisted $C^*$-algebras are matrix algebras
over noncommutative tori. For such systems we calculate the ordered $K$-theory and the
gauge-invariant semifinite traces of the resulting 3-graph $C^*$-algebras. We deduce that
every simple $C^*$-algebra of this form is Morita equivalent to the $C^*$-algebra of a
rank-2 Bratteli diagram in the sense of Pask-Raeburn-R{\o}rdam-Sims.
\end{abstract}

\maketitle

\section{Introduction}
Elliott's classification program, as described in \cite{ref9}, has been a very active
field of research in recent years. The program began with Elliott's classification of AF
algebras by their $K_0$-groups in \cite{Ell0}. Elliott subsequently expanded this
classification program to encompass all simple A$\TT$-algebras of real rank zero
\cite{Ell}, and then, in work with Gong, expanded it still further to encompass more
general AH algebras \cite{EllGon}, leading to the classification of simple AH-algebras of
slow dimension growth and of real rank zero \cite{ref11, ref20}. Paralleling these
results for stably finite $C^*$-algebras is the classification by $K$-theory of Kirchberg
algebras in the UCT class by Kirchberg and Phillips in the mid 1990s (see \cite{ref11,
ref20}).

Shortly after the introduction of graph $C^*$-algebras, it was shown in \cite{KumPasRae}
that every simple graph $C^*$-algebra is either purely infinite or AF, and so is
classified by $K$-theory either by the results of \cite{Ell0} or by those of \cite{ref11,
ref20}. As a result, the range of Morita-equivalence classes of simple $C^*$-algebras
that can be realised by graph $C^*$-algebras is completely understood. The introduction
of $k$-graphs and their $C^*$-algebras in \cite{KumPas} naturally raised the analogous
question. But it was shown in \cite{PasRaeRorSim} that there exist simple $k$-graph
algebras which are direct limits of matrix algebras over $C(\TT)$ and are neither AF nor
purely infinite. The examples constructed there are, nonetheless, classified by their
$K$-theory by \cite{Ell}, and the range of the invariant that they achieve is understood.

The general question of which simple $C^*$-algebras are Morita equivalent to $k$-graph
$C^*$-algebras is far from being settled, and the corresponding question for the twisted
$k$-graph $C^*$-algebras of \cite{KumPasSim2} is even less-well understood. In this paper
we consider a class of twisted $k$-graph $C^*$-algebras constructed using a procedure
akin to that in \cite{PasRaeRorSim}, except that the simple cycles there used to generate
copies of $M_n(C(\TT))$ are replaced here by $2$-dimensional simple cycles whose twisted
$C^*$-algebras are matrix algebras over noncommutative tori. When the noncommutative tori
all correspond to the same irrational rotation, we compute the ordered $K$-theory of
these examples by adapting Pimsner and Voiculescu's computation of the ordered $K$-theory
of the rotation algebras. Computing the ordered $K_0$-groups makes heavy use of traces on
the approximating subalgebras, and we finish by expanding on this to produce a detailed
analysis of traces on the $C^*$-algebras of rank-3 Bratteli diagrams.

Remarkably, it turns out that our construction does not expand the range of
Morita-equivalence classes of $C^*$-algebras obtained in \cite{PasRaeRorSim}: For every
rank-3 Bratteli diagram $E$ and every irrational $\theta$ such that the corresponding
twisted $C^*$-algebra $C^*(\Lambda_E, \dstar{}c^3_\theta)$ of the rank-3 Bratteli diagram
$\Lambda_E$ is simple, there is a rank-2 Bratteli diagram $\Gamma$ as in
\cite{PasRaeRorSim} whose $C^*$-algebra is Morita equivalent to $C^*(\Lambda_E,
\dstar{}c^3_\theta)$ (see Corollary~\ref{cor:range}). However, this requires both Elliott's
classification theorem and the Effros-Handelman-Shen characterisation of Riesz groups as
dimension groups. In particular, the rank-2 Bratteli diagram $\Gamma$ will depend heavily
on the value of $\theta$ as well as the diagram $E$.

\smallskip

In Section~\ref{sec3}, following a short introduction on twisted $(k+1)$-graph
$C^*$-algebras associated to covering sequences $(\Lambda_n, p_n)$ of $k$-graphs, we look
at the categorical cohomology of such $(k+1)$-graphs. We prove in Theorem \ref{cor2.6}
that each 2-cocycle on the $(k+1)$-graph $\Lambda$ associated to a covering sequence
$(\Lambda_n, p_n)$ is --- up to cohomology --- completely determined by its restriction
to the $k$-graph $\Lambda_1$.


In Section~\ref{sec5} we extend the notion of a covering sequence to a Bratteli diagram
of covering maps for a singly connected Bratteli diagram $E$. We construct a
$(k+1)$-graph $\Lambda$ from a Bratteli diagram of covering maps between $k$-graphs
$(\Lambda_v)_{v\in E^0}$ and --- upon fixing a 2-cocycle $c$ on $\Lambda$ --- we show how
to describe the twisted $(k+1)$-graph $C^*$-algebra $C^*(\Lambda, c)$, up to Morita
equivalence, as an inductive limit of twisted $k$-graph $C^*$-algebras, each of which is
a direct sum of $C^*$-algebras of the form $M_{n_v}(C^*(\Lambda_v, c|_{\Lambda_v}))$ (see
Theorem~\ref{thm5.6}).

In Section~\ref{sec:rank3BD}, we prove our main results. We consider a particular class
of 3-graphs associated to Bratteli diagrams of covering maps between rank-2 simple cycles; we call
these 3-graphs rank-3 Bratteli diagrams. We show in Theorem~\ref{cor5.8} how to compute
the ordered $K$-theory of twisted $C^*$-algebras of rank-3 Bratteli diagrams when the
twisting cocycle is determined by a fixed irrational angle $\theta$. We investigate when
such C*-algebras are simple in Corollary~\ref{cor:classifiable} and then prove in
Corollary~\ref{cor:range} in the presence of simplicity these C*-algebras can in fact be realised as
the $C^*$-algebras of rank-2 Bratteli diagrams in the sense of \cite{PasRaeRorSim}. In
Section~\ref{sec6} we briefly present some explicit examples of our $K$-theory
calculations. In Section~\ref{sec:BDtraces}, we describe an auxiliary AF algebra $C^*(F)$
associated to each rank-3 Bratteli diagram $\Lambda$, and exhibit an injection from
semifinite lower-semicontinuous traces on $C^*(F)$ to gauge-invariant semifinite
lower-semicontinuous traces of $C^*(\Lambda, c)$. We show that when $c$ is determined by
a fixed irrational rotation $\theta$, the map from traces on $C^*(F)$ to traces on
$C^*(\Lambda, c)$ is a bijection.

\section{Preliminaries and notation}

In this section we introduce the notion of $k$-graphs, covering sequences of $k$-graphs,
and twisted $k$-graph $C^*$-algebras which we can associate to covering sequences. These
are the main objects of study in this paper.

\subsection{\texorpdfstring{$k$}{k}-graphs}
Following \cite{KumPas, PasQuiRae, RaeSimYee} we briefly recall the notion of $k$-graphs.
For $k \ge 0$, a \emph{$k$-graph} is a nonempty countable small category equipped with a
functor $d\colon\Lambda \to \NN^k$ that satisfies the \emph{factorisation property}: for all
$\lambda \in \Lambda$ and $m,n\in \NN^k$ such that $d(\lambda) = m + n$ there exist
unique $\mu, \nu\in \Lambda$ such that $d(\mu) = m$, $d(\nu) = n$, and $\lambda =
\mu\nu$. When $d(\lambda) = n$ we say $\lambda$ has \emph{degree} $n$, and we write
$\Lambda^n = d^{-1}(n)$. The standard generators of $\NN^k$ are denoted $e_1, \dots
,e_k$, and we write $n_i$ for the $i^{\textrm{th}}$ coordinate of $n \in \NN^k$. For
$m,n\in \NN^k$, we write $m \vee n $ for their coordinate-wise maximum, and define a
partial order on $\NN^k$ by $m\leq n$ if $m_i\leq n_i$ for all $i$.

If $\Lambda$ is a $k$-graph, its \emph{vertices} are the elements of $\Lambda^0$. The
factorisation property implies that these are precisely the identity morphisms, and so
can be identified with the objects. For $\alpha \in \Lambda$ the \emph{source}
$s(\alpha)$ is the domain of $\alpha$, and the \emph{range} $r(\alpha)$ is the codomain
of $\alpha$ (strictly speaking, $s(\alpha)$ and $r(\alpha)$ are the identity morphisms
associated to the domain and codomain of $\alpha$).

For $u, v \in \Lambda^0$ and $E\subset \Lambda$, we write $uE := E \cap r^{-1}(u)$ and
$Ev := E \cap s^{-1}(v)$. For $n \in \NN^k$, we write
\[
\Lambda^{\leq n} := \{\lambda \in \Lambda : d(\lambda) \leq n\text{ and }
    s(\lambda)\Lambda^{e_i} = \emptyset \text{ whenever } d(\lambda) + e_i \leq n \}.
\]
We say that $\Lambda$ is \emph{connected} if the equivalence relation on $\Lambda^0$
generated by $\{(v,w)\in \Lambda^0 \times \Lambda^0 : v\Lambda w \neq \emptyset \}$ is
the whole of $\Lambda^0 \times \Lambda^0$. We say that $\Lambda$ is \emph{strongly
connected} if $v\Lambda w$ is nonempty for all $v,w\in \Lambda^0$. A \emph{morphism}
between $k$-graphs is a degree-preserving functor. We say that $\Lambda$ is
\emph{row-finite} if $v\Lambda^n$ is finite for all $v \in \Lambda^0$ and $n\in \NN^k$,
and $\Lambda$ is \emph{locally convex} if whenever $1 \leq i < j \leq k$, $e\in
\Lambda^{e_i}$, $f\in \Lambda^{e_j}$ and $r(e) = r(f)$, we can extend both $e$ and $f$ to
paths $ee'$ and $ff'$ in $\Lambda^{e_i+e_j}$. For $\lambda,\mu\in \Lambda$ we write
\[
    \Lambda^{\min}(\lambda,\mu):=\{(\alpha,\beta): \lambda\alpha = \mu\beta\text{ and }
        d(\lambda\alpha)=d(\lambda) \vee d(\mu)\}
\]
for the collection of pairs which give minimal common extensions of $\lambda$ and $\mu$.

A standard example of a $k$-graph is $T_k := \NN^k$ regarded as a $k$-graph with $d =
\id_{\NN^k}$.

\subsection{Covering sequences}

Following \cite{KumPasSim} a surjective morphism $p:\Gamma \to \Lambda$ between
$k$-graphs is a \emph{covering} if it restricts to bijections $\Gamma v \mapsto \Lambda
p(v)$ and $v\Gamma \mapsto p(v)\Lambda$ for $v \in \Gamma^0$. A covering $p:\Gamma \to
\Lambda$ is \emph{finite} if $p^{-1}(v)$ is finite for all $v\in \Lambda^0$.

\begin{defn}
A \emph{covering sequence} $(\Lambda_n, p_n)$ of $k$-graphs consists of $k$-graphs
$\Lambda_n$ and a sequence
\[
\xymatrix{\Lambda_1  &\Lambda_2 \ar[l]^{p_1} & \Lambda_3 \ar[l]^{p_2} & \ar[l]^{p_3} \dots }.
\]
of covering maps $p_n\colon \Lambda_{n+1}\to \Lambda_n$.
\end{defn}

Given a covering sequence $(\Lambda_n, p_n)$ of $k$-graphs
\cite[Corollary~2.11]{KumPasSim} shows that there exist a unique $(k+1)$-graph $\Lambda =
\lim_n (\Lambda_n, p_n)$, together with injective functors $\iota_n: \Lambda_n \to
\Lambda$ and a bijective map $e: \bigsqcup_{n\geq 2}\Lambda^0_{n}\to \Lambda^{e_{k+1}}$,
such that, identifying $\NN^{k+1}$ with $\NN^k \oplus \NN$,
\begin{enumerate}
\item $d(\iota_n(\lambda))=(d(\lambda),0)$, for $\lambda\in \Lambda_n$,
\item $\iota_m(\Lambda_m) \cap \iota_n(\Lambda_n) = \emptyset$, for $m\neq n$,
\item $\bigsqcup_{n\geq 1} \iota_n(\Lambda_n) = \{\lambda \in \Lambda :
    d(\lambda)_{k+1}=0\}$,
\item $s(e(v))=\iota_{n+1}(v)$, $r(e(v))=\iota_n(p_n(v))$, for $v\in
    \Lambda_{n+1}^0$, and
\item $e(r(\lambda))\iota_{n+1}(\lambda) = \iota_n(p_n(\lambda))e(s(\lambda))$, for
    $\lambda\in \Lambda_{n+1}$.
\end{enumerate}
We often suppress the inclusion maps $\iota_n$ and view the $\Lambda_n$ as subsets of
$\Lambda$. For $n>m$ we define $p_{m,n} := p_{m}\circ p_{m+1}\circ \cdots \circ
p_{n-2}\circ p_{n-1}: \Lambda_{n}\to \Lambda_{m}$; we define $p_{m,m} = \id_{\Lambda_m}$.

\subsection{Twisted \texorpdfstring{$k$}{k}-graph \texorpdfstring{$C^*$}{C*}-algebras}

Let $\Lambda$ be a row-finite locally convex $k$-graph and let $c\in Z^2(\Lambda,\TT)$. A
\emph{Cuntz-Krieger $(\Lambda, c)$-family} in a $C^*$-algebra $B$ is a function $s :
\lambda \mapsto s_\lambda$ from $\Lambda$ to $B$ such that
\begin{enumerate}
\item[(CK1)] $\{s_v : v\in \Lambda^0\}$ is a collection of mutually orthogonal
    projections;
\item[(CK2)] $s_\mu s_\nu = c(\mu,\nu)s_{\mu\nu}$ whenever $s(\mu) = r(\nu)$;
\item[(CK3)] $s^*_\lambda s_\lambda = s_{s(\lambda)}$ for all $\lambda\in \Lambda$;
    and
\item[(CK4)] $s_v = \sum_{\lambda \in v\Lambda^{\leq n}} s_\lambda s_\lambda^*$ for
    all $v\in \Lambda^0$ and $n\in \NN^k$.
\end{enumerate}
The following Lemma shows that our definition is consistent with previous literature.

\begin{lemma}
Let $\Lambda$ be a row-finite locally convex $k$-graph and let $c\in Z^2(\Lambda,\TT)$. A
function $s\colon\lambda \mapsto s_\lambda$ from $\Lambda$ to a $C^*$-algebra $B$ is a
Cuntz-Krieger $(\Lambda, c)$-family if and only if it is a Cuntz-Krieger $(\Lambda,
c)$-family in the sense of \cite{Whi}.
\end{lemma}
\begin{proof}
Recall that $E \subseteq v\Lambda$ is \emph{exhaustive} if for every $\lambda \in
v\Lambda$ there exists $\mu \in E$ such that $\Lambda^{\min}(\lambda,\mu) \not=
\emptyset$. According to \cite{SimWhiWhi}, a Cuntz-Krieger $(\Lambda, c)$-family is a
function $s:\lambda \mapsto s_\lambda$ such that
\begin{enumerate}
\item[(TCK1)] $\{s_v : v\in \Lambda^0\}$ is a collection of mutually orthogonal
    projections;
\item[(TCK2)] $s_\mu s_\nu = c(\mu,\nu)s_{\mu\nu}$ whenever $s(\mu) = r(\nu)$;
\item[(TCK3)] $s^*_\lambda s_\lambda = s_{s(\lambda)}$ for all $\lambda\in \Lambda$;
\item[(TCK4)] $s_\mu s^*_\mu s_\nu s^*_\nu=\sum_{(\alpha,\beta)\in
    \Lambda^{\min}(\lambda,\mu)} s_{\lambda\alpha} s_{\lambda\alpha}^*$ for all
    $\mu,\nu \in \Lambda$; and
\item[(CK)] $\prod_{\lambda\in E}(s_v-s_\lambda s_\lambda^*)=0$ for all $v\in
    \Lambda^0$ and finite \emph{exhaustive} $E\subseteq v\Lambda$.
\end{enumerate}
So we must show that (CK1)--(CK4) are equivalent to (TCK1)--(CK).

First suppose that $s$ satisfies (CK1)--(CK4). Then it clearly satisfies (TCK1)--(TCK3).
For~(TCK4), fix $\lambda,\mu \in \Lambda$. It suffices to show that
$s_\lambda^*s_\mu=\sum_{(\alpha,\beta)\in \Lambda^{\min}(\lambda,\mu)}
\overline{c(\lambda,\alpha)}c(\mu,\beta) s_\alpha s_\beta^*$ (see
\cite[Lemma~3.1.5]{Whi}). Define $N:=d(\lambda)\vee d(\mu)$ and $E':=\{(\alpha,\beta):
\lambda\alpha=\mu\beta, \lambda\alpha \in r(\lambda)\Lambda^{\leq N}\}$. The proof of
\cite[Proposition 3.5]{RaeSimYee} gives $s_\lambda^*s_\mu=\sum_{(\alpha,\beta)\in
E'}\overline{c(\lambda,\alpha)}c(\mu,\beta) s_\alpha s_\beta^*$. Clearly
$\Lambda^{\min}(\lambda,\mu) \subseteq E'$ since $r(\lambda)\Lambda^N\subseteq
r(\lambda)\Lambda^{\leq N}$. Conversely for any $(\alpha,\beta)\in E'$ we have
$\lambda\alpha = \mu\beta$ by definition, and hence $d(\lambda\alpha) \ge N$. Since
$\lambda\alpha \in \Lambda^{\le N}$, we also have $d(\lambda\alpha) \le N$, and hence we
have equality. Thus $E' = \Lambda^{\min}(\lambda,\mu)$. This gives~(TCK4).

For~(CK), fix $v\in \Lambda^0$ and a finite exhaustive $E\subseteq v\Lambda$. With
$N:=\bigvee_{\lambda\in E} d(\lambda)$ and $E':=\{\lambda\nu: \lambda\in E, \nu\in
s(\lambda)\Lambda^{\leq N-d(\lambda)}\}$ we have $E'=v\Lambda^{\leq N}$ by an induction
using \cite[Lemma~3.12]{RaeSimYee}. Relation~(TCK4) implies that the $s_\lambda
s^*_\lambda$ where $\lambda \in E$ commute, and that the $s_{\lambda\nu}
s^*_{\lambda\nu}$ are mutually orthogonal. Also, (TCK2) implies that $s_v - s_\lambda
s^*_\lambda \le s_v - s_{\lambda\nu} s^*_{\lambda\nu}$ for all $\lambda\nu \in E'$, and
so
\[
\prod_{\lambda\in E}(s_v-s_\lambda s_\lambda^*)
    \leq \prod_{\mu\in E'}(s_v-s_\mu s_\mu^*)
    = s_v-\sum_{\mu\in v\Lambda^{\leq N}} s_\mu s_\mu^*
    =0.
\]

Now suppose that $s$ satisfies (TCK1)--(CK). Then it clearly satisfies (CK1)--(CK3).
For~(CK4), fix $v\in \Lambda^0$ and $1\leq i\leq k$ with $E:=v\Lambda^{e_i}\neq
\emptyset$. Then~(TCK4) implies that the $s_\lambda s^*_\lambda$ for $\lambda \in E$ are
mutually orthogonal. Since $E\subseteq v\Lambda$ is finite and exhaustive~(CK) then gives
\[
    0=\prod_{\lambda\in E}(s_v-s_\lambda s_\lambda^*)
        =s_v-\sum_{\lambda\in E}s_\lambda s_\lambda^*.\qedhere
\]
\end{proof}

The \emph{twisted $k$-graph $C^*$-algebra $C^*(\Lambda, c)$} is the universal
$C^*$-algebra generated by a Cuntz-Krieger $(\Lambda, c)$-family.

\section{Covering sequences and cohomology}\label{sec3}

We will be interested in twisted $C^*$-algebras associated to 3-graphs analogous to the
rank-2 Bratteli diagrams of \cite{PasRaeRorSim} (see Section~\ref{sec:rank3BD}). The
building blocks for these 3-graphs are covering systems of $k$-graphs. In this section we
investigate their second cohomology groups. Our results substantially simplify the
problem of studying the associated $C^*$-algebras later because it allows us to assume
the the twisting $2$-cocycles are pulled back from a cocycle of a standard form on
$\ZZ^2$  (see Remark~\ref{rmk:easycocycles}).

We briefly recap the categorical cohomology of $k$-graphs (see \cite{KumPasSim2}). Let
$\Lambda$ be a $k$-graph, and let $A$ be an abelian group. For each integer $r \geq 1$,
let $\Lambda^{*r}:=\{(\lambda_1, \dots, \lambda_r)\in \prod_{i=1}^r \Lambda :
s(\lambda_i)=r(\lambda_{i+1}) \text{ for each }i\}$ be the collection of
\emph{composable} $r$-tuples in $\Lambda$, and let $\Lambda^{*0}:=\Lambda^0$. For $r \geq
1$, a function $f: \Lambda^{*r}\to A$ is an \emph{$r$-cochain} if $f(\lambda_1, \dots ,
\lambda_r) = 0$ whenever $\lambda_i \in \Lambda^0$ for some $i \leq r$. A $0$-cochain is
any function $f\colon\Lambda^0 \to A$. We write $C^r(\Lambda,A)$ for the group of all
$r$-cochains under pointwise addition. Define maps $\delta^r\colon C^r(\Lambda, A) \to
C^{r+1}(\Lambda, A)$ by $\delta^0(f)(\lambda) = f(s(\lambda)) - f(r(\lambda))$ and
\begin{align*}
\delta^r(f)(\lambda_0, \dots, \lambda_r)
    &= f(\lambda_1, \dots, \lambda_r) \\
    &\qquad + \sum^r_{i=1} (-1)^i f(\lambda_0, \dots, \lambda_{i-2}, \lambda_{i-1}\lambda_i,
        \lambda_{i+1}, \dots, \lambda_r) \\
    &\qquad\qquad + (-1)^{r+1} f(\lambda_0, \dots, \lambda_{r-1})\quad\text{ for $r \ge 1$.}
\end{align*}
Let $B^r(\Lambda, A) := \operatorname{im}(\delta^{r-1})$ and $Z^r(\Lambda, A) =
\ker(\delta^r)$. A calculation \cite[(3.3)--(3.5)]{KumPasSim2} shows that $B^r(\Lambda,
A) \subseteq Z^r(\Lambda, A)$. We define $H^r(\Lambda, A) := Z^r(\Lambda, A)/B^r(\Lambda,
A)$. We call the elements of $B^r(\Lambda, A)$ \emph{$r$-coboundaries}, and the elements
of $Z^r(\Lambda, A)$ \emph{$r$-cocycles}. A 2-cochain $c\in C^2(\Lambda,A)$ is a
2-cocycle if and only if it satisfies the {\em cocycle identity}
$c(\lambda,\mu)+c(\lambda\mu,\nu)=c(\mu,\nu)+c(\lambda,\mu\nu)$.

As a notational convention, if $\Gamma \subseteq \Lambda$ is a subcategory and $c \in
Z^r(\Lambda, A)$ then we write $c|_\Gamma$, rather than $c|_{\Gamma^{*r}}$ for the
restriction of $c$ to the composable $r$-tuples of $\Gamma$. If $\Lambda$ and $\Gamma$
are $k$-graphs and $\phi\colon\Lambda \to \Gamma$ is a functor, and if $c \in Z^2(\Gamma,
A)$, then $\phi_* c\colon\Lambda^{*2} \to A$ is defined by $\phi_* c(\lambda,\mu) =
c(\phi(\lambda), \phi(\mu))$.

In \cite{KumPasSim2}, the categorical cohomology groups described above were decorated
with an underline to distinguish them from the cubical cohomology groups of
\cite{KumPasSim3}. In this paper, we deal only with categorical cohomology, so we have chosen
to omit the underlines.

\begin{defn}\label{def2.2}
Let $\Lambda=\lim_n (\Lambda_n, p_n)$ be the $(k+1)$-graph associated to a covering
sequence of $k$-graphs. A sequence $(c_n)$ of cocycles $c_n\in Z^2(\Lambda_n, A)$ is
\emph{compatible} if there is a 2-cocycle $c\in Z^2(\Lambda, A)$ such that
$c|_{\Lambda_n} = c_n$ for $n\geq 1$; we say that the $c_n$ are compatible \emph{with
respect to $c$}.
\end{defn}

Let $\Lambda=\lim_n (\Lambda_n, p_n)$ be the $(k+1)$-graph associated to a covering
sequence of $k$-graphs. For each $v \in \Lambda^0$ there is a unique element $\xi_v$ of
$\Lambda_1^0 \Lambda^{\NN e_{k+1}} v$; if $v \in \Lambda_n^0$, then $r(\xi_v) =
p_{1,n}(v)$. The factorisation property implies that for each $\lambda \in \Lambda$ there
is a unique factorisation
\begin{equation}\label{eq:pi def}
\xi_{r(\lambda)}\lambda = \pi(\lambda)\beta\text{ with $\pi(\lambda) \in \Lambda_1$ and
    $\beta \in \Lambda^{\NN e_{k+1}}$.}
\end{equation}
We call the assignment $\lambda \mapsto \pi(\lambda)$ the \emph{projection} of $\Lambda$
onto $\Lambda_{1}$.

Observe that if $\lambda \in \Lambda^n$, then $\pi(\lambda) = p_{1,n}(\lambda)$, and if
$\lambda \in \Lambda^{\NN e_{k+1}}$ with $r(\lambda) \in \Lambda_n^0$, then $\pi(\lambda)
= p_{1,n}(r(\lambda))$. In particular, $\pi(\xi_v) = r(\xi_v) = p_{1,n}(v)$ for $v \in
\Lambda_n^0$. In general, if $r(\lambda) \in \Lambda_n^0$, then we can factorise $\lambda
= \lambda'\lambda''$ with $\lambda' \in \Lambda_n$ and $\lambda'' \in \Lambda^{\NN
e_{k+1}}$, and then $\pi(\lambda) = p_{1,n}(\lambda')$.

\begin{lemma}\label{lem3.4}
Let $\Lambda = \lim_n(\Lambda_n, p_n)$ be the $(k+1)$-graph associated to a covering
sequence of $k$-graphs. The projection $\pi$ of $\Lambda$ onto $\Lambda_{1}$
of~\eqref{eq:pi def} is a functor. The formula $\pi_*c(\lambda_1, \lambda_2) =
c(\pi(\lambda_1), \pi(\lambda_2))$ for $(\lambda_1, \lambda_2) \in \Lambda^{*2}$
determines a homomorphism $\pi_* \colon Z^2(\Lambda_1,A) \to Z^2(\Lambda,A)$.
\end{lemma}
\begin{proof}
Fix $(\lambda,\mu)\in \Lambda^{*2}$ and factorise $\lambda=\nu\alpha$ and
$\mu=\beta\gamma$ where $\nu\in\Lambda_n$, $\gamma\in\Lambda_m$, and $\alpha,\beta\in
\Lambda^{\NN e_{k+1}}$. The factorisation property gives $\pi(\lambda\mu) =
p_{1,n}(\nu)p_{1,m}(\gamma) = \pi(\lambda)\pi(\mu)$:
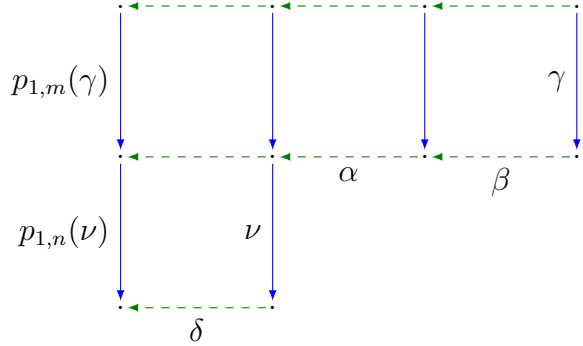
\[
\begin{tikzpicture}[scale=2.0]
    \node[inner sep=0.5pt, circle] (11) at (0,0) {$.$};
    \node[inner sep=0.5pt, circle] (12) at (1,0) {$.$};
    \node[inner sep=0.5pt, circle] (13) at (2,0) {$.$};
    \node[inner sep=0.5pt, circle] (14) at (3,0) {$.$};
    \node[inner sep=0.5pt, circle] (15) at (0,-1) {$.$};
    \node[inner sep=0.5pt, circle] (21) at (0,1) {$.$};
    \node[inner sep=0.5pt, circle] (22) at (1,1) {$.$};
    \node[inner sep=0.5pt, circle] (23) at (2,1) {$.$};
    \node[inner sep=0.5pt, circle] (24) at (3,1) {$.$};
    \node[inner sep=0.5pt, circle] (25) at (1,-1) {$.$};
    \draw[-latex, dashed, green!50!black] (12)--(11);
    \draw[-latex, dashed, green!50!black] (13)--(12) node[pos=0.5, below, black] {$\alpha$};
    \draw[-latex, green!50!black, dashed] (14)--(13) node[pos=0.5, below, black] {$\beta$};
    \draw[-latex, green!50!black, dashed] (25)--(15) node[pos=0.5, below, black] {$\delta$};
    \draw[-latex, green!50!black, dashed] (22)--(21);
    \draw[-latex, green!50!black, dashed] (23)--(22);
    \draw[-latex, green!50!black, dashed] (24)--(23);
    \draw[-latex, blue] (21)--(11) node[pos=0.5, left, black] {$p_{1,m} (\gamma )$};
    \draw[-latex, blue] (22)--(12);
    \draw[-latex, blue] (23)--(13);
    \draw[-latex, blue] (24)--(14) node[pos=0.5, left, black] {$\gamma$};
    \draw[-latex, blue] (12)--(25) node[pos=0.5, left, black] {$\nu$};
	\draw[-latex, blue] (11)--(15) node[pos=0.5, left, black] {$p_{1,n} ( \nu )$};
\end{tikzpicture}
\]
Hence $\pi$ is a functor.

Since functors send identity morphisms to identity morphisms, it follows immediately that
$\pi_*c$ is a 2-cocycle. Since the operations in the cohomology groups are pointwise,
$\pi_*$ is a homomorphism.
\end{proof}

Using the covering maps between the $k$-graphs in a covering sequence we can build a
compatible sequence of 2-cocycles from a 2-cocycle on $\Lambda_1$.

\begin{thm}\label{thm2.1}
Let $\Lambda=\lim_n (\Lambda_n, p_n)$ be the $(k+1)$-graph associated to a covering
sequence of $k$-graphs, and let $c\in Z^2(\Lambda_1, A)$. Let $c_n := (p_{1,n})_* c$ for
$n \ge 1$. Then each $c_n \in Z^2(\Lambda_n, A)$, and the $c_n$ are compatible with
respect to $\overline{c} := \pi_*c$.
\end{thm}
\begin{proof}
For $\lambda\in \Lambda_n$, repeated use of property~(5) of $\Lambda$ shows that
$\pi(\lambda)=p_{1,n}(\lambda)$. Hence
$\overline{c}|_{\Lambda_n}=(\pi_{|_{\Lambda_n}})_*c = c_n$ for $n\geq 1$.
Lemma~\ref{lem3.4} implies that $\overline{c}$ is a 2-cocycle on $\Lambda$. Since the
restriction of a 2-cocycle is again a 2-cocycle it follows that
$c_n=\overline{c}|_{\Lambda_n}$ is a 2-cocycle on $\Lambda_n$.
\end{proof}

Theorem~\ref{thm2.1} provides a map $c\mapsto \overline{c}$ from 2-cocycles on
$\Lambda_1$ to 2-cocycles on $\Lambda$. It turns out that this is essentially the only
way to construct 2-cocycles on $\Lambda$ (we will make this more precise in
Theorem~\ref{cor2.6}).

\begin{lemma}\label{cor2.5}
Let $\Lambda=\lim_n (\Lambda_n, p_n)$ be the $(k+1)$-graph associated to a
covering sequence of $k$-graphs, and let $c\in Z^2(\Lambda_1, A)$. There exists a unique
2-cocycle $c'$ on $\Lambda$ extending $c$ such that
\begin{equation}
\label{eqn3.1}
c'(\lambda,\mu)= 0\quad\text{ whenever $\lambda \in\Lambda^{\NN e_{k+1}}$ or $\mu \in\Lambda^{\NN e_{k+1}}$.}
\end{equation}
\end{lemma}
\begin{proof}
Theorem~\ref{thm2.1} implies that $\overline{c}$ satisfies $\overline{c}|_{\Lambda_1} =
c$, and since $\pi(\lambda) \in \Lambda_1^0$ whenever $\lambda \in \Lambda^{\NN
e_{k+1}}$,
\[
\overline{c}(\lambda,\mu)
    = c(\pi(\lambda),\pi(\mu))=0 \quad\text{ whenever $\lambda \in\Lambda^{\NN e_{k+1}}$ or
        $\mu \in\Lambda^{\NN e_{k+1}}$.}
\]

Now suppose that $c' \in Z^2(\Lambda, A)$ satisfies~\eqref{eqn3.1}.

We claim first that
\begin{equation}
\label{eqn3.2}
c'(\alpha\lambda,\mu\beta)=c'(\lambda,\mu)
\end{equation}
whenever $\alpha,\beta \in \Lambda^{\NN
e_{k+1}}$ and $\lambda,\mu \in \Lambda_n$. To see this,
observe that~\eqref{eqn3.1} gives
\[
c'(\alpha,\lambda) = 0,\quad  c'(\alpha,\lambda\mu\beta) = 0, \quad  c'(\mu,\beta) = 0\quad\text{and}\quad
c'(\lambda\mu,\beta) = 0.
\]
Repeated application of the cocycle identity gives
\begin{align*}
c'(\alpha\lambda,\mu\beta)
    &= c'(\alpha\lambda,\mu\beta) + c'(\alpha,\lambda) + c'(\mu,\beta) \\
    &= c'(\alpha,\lambda\mu\beta) + c'(\lambda,\mu\beta) + c'(\mu,\beta)
    = c'(\lambda,\mu) + c'(\lambda\mu,\beta)
    = c'(\lambda,\mu).
\end{align*}

We now claim that
\begin{equation}
\label{eqn3.3}
c'(\lambda,\mu)=c'(p_n(\lambda),p_n(\mu)),
\end{equation}
for composable $\lambda, \mu\in \Lambda_{n+1}$. To see this, use property~(5) of $\Lambda$ to find
$\alpha,\beta,\gamma$ in $\Lambda^{e_{k+1}}$ such that $\alpha\lambda =
p_n(\lambda)\beta$ and $\beta\mu=p_n(\mu)\gamma$:
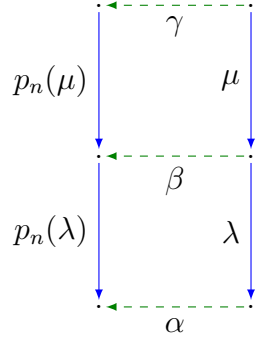
\[
\begin{tikzpicture}[scale=2.0]
    \node[inner sep=0.5pt, circle] (11) at (0,0) {$.$};
    \node[inner sep=0.5pt, circle] (12) at (1,0) {$.$};
    \node[inner sep=0.5pt, circle] (15) at (0,-1) {$.$};
    \node[inner sep=0.5pt, circle] (21) at (0,1) {$.$};
    \node[inner sep=0.5pt, circle] (22) at (1,1) {$.$};
    \node[inner sep=0.5pt, circle] (25) at (1,-1) {$.$};
    \draw[-latex, dashed, green!50!black] (12)--(11) node[pos=0.5, below, black] {$\beta$};
    \draw[-latex, green!50!black, dashed] (25)--(15) node[pos=0.5, below, black] {$\alpha$};
    \draw[-latex, green!50!black, dashed] (22)--(21) node[pos=0.5, below, black] {$\gamma$};
    \draw[-latex, blue] (21)--(11) node[pos=0.5, left, black] {$p_{n} (\mu )$};
    \draw[-latex, blue] (22)--(12) node[pos=0.5, left, black] {$\mu$};
    \draw[-latex, blue] (12)--(25) node[pos=0.5, left, black] {$\lambda$};
	\draw[-latex, blue] (11)--(15) node[pos=0.5, left, black] {$p_{n} ( \lambda )$};
\end{tikzpicture}
\]
Then
\[
c'(p_n(\lambda)\beta,\mu) + c'(p_n(\lambda),\beta)
    = c'(p_n(\lambda),\beta\mu) + c'(\beta,\mu),
\]
and so~\eqref{eqn3.1} gives $c'(p_n(\lambda)\beta,\mu) = c'( c'(p_n(\lambda),\beta\mu)$.
Now~\eqref{eqn3.2} shows that
\begin{align*}
c'(p_n(\lambda),p_n(\mu))
    &= c'(p_n(\lambda),p_n(\mu)\gamma)
    = c'(p_n(\lambda),\beta\mu)\\
    &= c'(p_n(\lambda)\beta,\mu)
    = c'(\alpha\lambda,\mu)
    = c'(\lambda,\mu).
\end{align*}

We now show that $c'=\overline{c}$. We have seen that both $c'$ and $\overline{c}$
satisfy~\eqref{eqn3.1}, and so they both satisfy~\eqref{eqn3.2}. It therefore suffices to
show that $c'|_{\Lambda_n} = \overline{c}|_{\Lambda_n}$ for each $n$. Fix composable
$\lambda,\mu \in \Lambda^n$. We have $\overline{c}(\lambda,\mu) = c(p_{1,n}(\lambda),
p_{1,n}(\mu))$ by definition. Repeated applications of~\eqref{eqn3.3} give
$c'(\lambda,\mu) = c'(p_{1,n}(\lambda), p_{1,n}(\mu))$. Since $c'$ extends $c$, we deduce
that $c'(\lambda,\mu) = \overline{c}(\lambda,\mu)$.
\end{proof}

\begin{lemma}\label{thm2.5}
Let $\Lambda=\lim_n (\Lambda_n, p_n)$ be the $(k+1)$-graph associated to a covering
sequence of $k$-graphs, and let $c\in Z^2(\Lambda, A)$. Let $\overline{c|_{\Lambda_1}} =
\pi_* (c|_{\Lambda_1})$ as in Theorem~\ref{thm2.1}. Then there exists $b\in
C^1(\Lambda,A)$ such that
\[
c-\delta^1b = \overline{c|_{\Lambda_1}}.
\]
\end{lemma}
\begin{proof}
For $v\in\Lambda^0$ let $\xi_v$ be the unique element of $\Lambda_1^0 \Lambda^{\NN
e_{k+1}} v$. For $\lambda \in \Lambda$ define
\[
b(\lambda)
    = c(\xi_{r(\lambda)},\lambda) - c(\pi(\lambda),\xi_{s(\lambda)}).
\]
If $\lambda \in \Lambda^0$ then $\pi(\lambda) \in \Lambda^0$ as well, and so $b(\lambda)
= 0$. So $b\in C^1(\Lambda,A)$.

Since the restriction of the maps $b$ and $\delta^1b(\lambda,\mu) =
b(\lambda)+b(\mu)-b(\lambda\mu)$ to $\Lambda_1$ are identically zero, the cocycle
$c':=c-\delta^1b$ extends $c$. To conclude that $c'=\overline{c|_{\Lambda_1}}$ it now
suffices, by Lemma~\ref{cor2.5}, to verify that $c'$ satisfies~\eqref{eqn3.1}.

We first prove that
\begin{equation}\label{eqn3.4}
c'(\lambda, \mu)= 0\text{ whenever $\lambda \in\Lambda^{\NN e_{k+1}}$.}
\end{equation}
Fix such a composable pair $\lambda,\mu\in\Lambda$, and factorise $\mu=\eta\beta'$ with
$\eta\in\Lambda_n$ and $\beta' \in \Lambda^{\NN e_{k+1}}$. Let $l$ be the integer such
that $r(\lambda) \in \Lambda^0_l$. By property~(5) of $\Lambda$ there exist
$\lambda',\beta$ in $\Lambda^{\NN e_{k+1}}$ with $r(\lambda') \in \Lambda^0_l$ and
$r(\beta) \in \Lambda^0_1$, and $\gamma\in \Lambda_m^{d(\eta)}$ such that
$\xi_{r(\lambda)}\lambda\beta\gamma = p_{1,n}(\eta)\xi_{r(\lambda')}\lambda'\beta'$:
\[
\begin{tikzpicture}[scale=2.0]
    \node[inner sep=0.5pt, circle] (11) at (0,0) {$.$};
    \node[inner sep=0.5pt, circle] (12) at (1,0) {$.$};
    \node[inner sep=0.5pt, circle] (13) at (2,0) {$.$};
    \node[inner sep=0.5pt, circle] (14) at (3,0) {$.$};
    \node[inner sep=0.5pt, circle] (21) at (0,1) {$.$};
    \node[inner sep=0.5pt, circle] (22) at (1,1) {$.$};
    \node[inner sep=0.5pt, circle] (23) at (2,1) {$.$};
    \node[inner sep=0.5pt, circle] (24) at (3,1) {$.$};
    \draw[-latex, dashed, green!50!black] (12)--(11) node[pos=0.5, below, black] {$\xi_{r(\lambda)}$};
    \draw[-latex, dashed, green!50!black] (13)--(12) node[pos=0.5, below, black] {$\lambda$};
    \draw[-latex, green!50!black, dashed] (14)--(13) node[pos=0.5, below, black] {$\beta$};
    \draw[-latex, green!50!black, dashed] (22)--(21) node[pos=0.5, below, black] {$\xi_{r(\lambda')}$};
    \draw[-latex, green!50!black, dashed] (23)--(22) node[pos=0.5, below, black] {$\lambda'$};
    \draw[-latex, green!50!black, dashed] (24)--(23) node[pos=0.5, below, black] {$\beta'$};
    \draw[-latex, blue] (21)--(11) node[pos=0.5, left, black] {$p_{1,n} (\eta )$};
    \draw[-latex, blue] (22)--(12);
    \draw[-latex, blue] (23)--(13) node[pos=0.5, left, black] {$\eta$};
    \draw[-latex, blue] (24)--(14) node[pos=0.5, left, black] {$\gamma$};
\end{tikzpicture}
\]
We have $\pi(\mu) = \pi(\eta) = p_{1,n}(\eta) = \pi(\gamma) = p_{1,m}(\gamma)$. We
prove~\eqref{eqn3.4} in three steps.

(1) First we show that $c'(\beta,\gamma)=0$ and $c'(\lambda\beta,\gamma)=0$. The cocycle
identity gives $c(\xi_{r(\beta)}, \beta\gamma) + c(\beta,\gamma) = c(\xi_{r(\beta)}\beta,
\gamma) + c(\xi_{r(\beta)}, \beta)$. Since $\xi_{r(\beta)}\beta=\xi_{r(\gamma)}$, the
definition of $b$ gives
\begin{align*}
c'(\beta,\gamma)&=c(\beta,\gamma) - b(\beta) - b(\gamma) + b(\beta\gamma)\\
&=c(\beta,\gamma) - \Big(c(\xi_{r(\beta)},\beta) - 0\Big)
    - \Big(c(\xi_{r(\gamma)}, \gamma) - c(p_{1,n}(\gamma), \xi_{s(\gamma)})\Big) \\
    &\qquad\qquad{}
    + \Big(c(\xi_{r(\beta)}, \beta\gamma) - c(p_{1,n}(\gamma), \xi_{s(\gamma)})\Big)\\
&=\Big(c(\xi_{r(\beta)}, \beta\gamma) + c(\beta, \gamma)\Big)
    - \Big(c(\xi_{r(\beta)}\beta, \gamma) + c(\xi_{r(\beta)}, \beta)\Big) = 0.
\end{align*}
Applying this calculation to $\lambda\beta$ rather than $\beta$ gives
$c'(\lambda\beta,\gamma) = 0$ as well.

(2) Next we show that $c'(\lambda,\beta)=0$. We have $c'(\xi_{r(\lambda)}, \lambda\beta)
+ c'(\lambda,\beta) = c'(\xi_{r(\lambda)}\lambda, \beta) + c'(\xi_{r(\lambda)},
\lambda)$. Since $\xi_{r(\lambda)}\lambda = \xi_{r(\beta)}$ it follows that
\begin{align*}
c'(\lambda,\beta)
    & = c(\lambda,\beta)-b(\lambda)-b(\beta)+b(\lambda\beta)\\
    & = c(\lambda,\beta)-\Big(c(\xi_{r(\lambda)},\lambda)-0\Big)
        - \Big(c(\xi_{r(\beta)},\beta)-0\Big)
        + \Big(c(\xi_{r(\lambda)},\lambda\beta)-0\Big)\\
    & = \Big(c(\xi_{r(\lambda)},\lambda\beta) + c(\lambda,\beta)\Big)
        -\Big(c(\xi_{r(\lambda)}\lambda,\beta) + c(\xi_{r(\lambda)},\lambda)\Big) = 0.
\end{align*}

(3) Finally, to establish~\eqref{eqn3.4}, we apply the cocycle identity
$c'(\lambda\beta,\gamma) + c'(\lambda,\beta) = c'(\lambda,\beta\gamma) +
c'(\beta,\gamma)$ and steps (1)~and~(2) to see that
\[
c'(\lambda,\mu)
    = c'(\lambda,\eta\beta')
    = c'(\lambda,\beta\gamma)
    = 0.
\]

It remains to show that
\begin{equation}
\label{eqn3.5}
c'(\lambda,\mu)= 0\quad\text{ whenever $\mu\in\Lambda^{\NN e_{k+1}}$.}
\end{equation}
Fix such a composable pair $\lambda,\mu\in \Lambda$, and factorise $\lambda=\alpha\eta$
with $\alpha \in \Lambda^{\NN e_{k+1}}$ and $\eta\in \Lambda_n$. Using the factorisation
property, we obtain $\alpha' \in \Lambda^{d(\alpha)}$, $\mu' \in \Lambda^{d(\mu)}$ and
$\gamma \in \Lambda_m^{d(\eta)}$ that make the following diagram commute.
\[
\begin{tikzpicture}[scale=2.0]
    \node[inner sep=0.5pt, circle] (11) at (0,0) {$.$};
    \node[inner sep=0.5pt, circle] (12) at (1,0) {$.$};
    \node[inner sep=0.5pt, circle] (13) at (2,0) {$.$};
    \node[inner sep=0.5pt, circle] (14) at (3,0) {$.$};
    \node[inner sep=0.5pt, circle] (21) at (0,1) {$.$};
    \node[inner sep=0.5pt, circle] (22) at (1,1) {$.$};
    \node[inner sep=0.5pt, circle] (23) at (2,1) {$.$};
    \node[inner sep=0.5pt, circle] (24) at (3,1) {$.$};
    \draw[-latex, dashed, green!50!black] (12)--(11) node[pos=0.5, below, black] {$\xi_{r(\lambda)}$};
    \draw[-latex, dashed, green!50!black] (13)--(12) node[pos=0.5, below, black] {$\alpha$};
    \draw[-latex, green!50!black, dashed] (14)--(13) node[pos=0.5, below, black] {$\mu'$};
    \draw[-latex, green!50!black, dashed] (22)--(21) node[pos=0.5, below, black] {$\xi_{r(\alpha')}$};
    \draw[-latex, green!50!black, dashed] (23)--(22) node[pos=0.5, below, black] {$\alpha'$};
    \draw[-latex, green!50!black, dashed] (24)--(23) node[pos=0.5, below, black] {$\mu$};
    \draw[-latex, blue] (21)--(11) node[pos=0.5, left, black] {$p_{1,n} (\eta )$};
    \draw[-latex, blue] (22)--(12);
    \draw[-latex, blue] (23)--(13) node[pos=0.5, left, black] {$\eta$};
    \draw[-latex, blue] (24)--(14) node[pos=0.5, left, black] {$\gamma$};
\end{tikzpicture}
\]

Equation~\eqref{eqn3.4} gives $c'(\alpha,\eta)=0$ and $c'(\alpha,\eta\mu)=0$. By the
cocycle identity $c'(\alpha\eta,\mu) + c'(\alpha,\eta) = c'(\alpha,\eta\mu) +
c'(\eta,\mu)$. So we need only check that $c'(\eta,\mu) = 0$. We consider the cocycle
identity
\[
c(p_{1,m}(\gamma),\xi_{r(\mu)}\mu) + c(\xi_{r(\mu)},\mu)
    = c(p_{1,m}(\gamma)\xi_{r(\mu)},\mu) + c(p_{1,m}(\gamma),\xi_{r(\mu)}).
\]
Since $\xi_{s(\gamma)} = \xi_{r(\mu)}\mu$, we have $p_{1,n}(\eta)=p_{1,m}(\gamma)$. Since
$\xi_{s(\eta)} = \xi_{r(\mu)}$ and $p_{1,m}(\gamma)\xi_{r(\mu)} = \xi_{r(\mu')}\eta$, we
obtain
\begin{equation}
\label{eqn3.6}
c(p_{1,n}(\eta), \xi_{s(\eta)}) - c(\xi_{r(\mu)}, \mu) - c(p_{1,m}(\gamma),\xi_{s(\gamma)})
    = - c(\xi_{r(\mu')}\eta, \mu).
\end{equation}
Hence
\begin{align*}
c'(\eta,\mu)
    &= c(\eta,\mu) - b(\eta) - b(\mu) + b(\eta\mu)\\
    &= c(\eta, \mu) - b(\eta) - b(\mu) + b(\mu'\gamma)\\
    &= c(\eta, \mu)
        - \Big(c(\xi_{r(\eta)}, \eta) - c(p_{1,n}(\eta), \xi_{s(\eta)})\Big) \\
        &\qquad{} - \Big(c(\xi_{r(\mu)},\mu) - 0\Big)
            + \Big(c(\xi_{r(\mu')}, \mu'\gamma) - c(p_{1,m}(\gamma),\xi_{s(\gamma)})\Big)\\
    &= c(\eta,\mu) - c(\xi_{r(\eta)}, \eta)
        + c(\xi_{r(\mu')}, \mu'\gamma) - c(\xi_{r(\mu')}\eta, \mu)\quad\text{ by~\eqref{eqn3.6}}\\
    &= c(\eta,\mu) - c(\xi_{r(\mu')}, \eta)
        + c(\xi_{r(\mu')}, \eta\mu) - c(\xi_{r(\mu')}\eta, \mu) = 0,
\end{align*}
establishing~\eqref{eqn3.5}.
\end{proof}

\begin{thm}\label{cor2.6}
Let $\Lambda=\lim_n (\Lambda_n, p_n)$ be the $(k+1)$-graph associated to a covering
sequence of $k$-graphs. The restriction map $c\mapsto c|_{\Lambda_1}$ from
$Z^2(\Lambda,A)$ to $Z^2(\Lambda_1, A)$ induces an isomorphism $H^2(\Lambda,A)\cong
H^2(\Lambda_1,A)$.
\end{thm}
\begin{proof}
Surjectivity follows from Lemma~\ref{cor2.5}. To verify injectivity fix 2-cocycles
$c_1,c_2$ on $\Lambda$ such that $c_1|_{\Lambda_1}-c_2|_{\Lambda_1}=\delta^1 b$ for some
$b\in C^1(\Lambda_1,A)$. Lemma~\ref{thm2.5} gives $b_1,b_2\in C^1(\Lambda,A)$ such that
\[
c_1 - \delta^1b_1 = \pi_*(c_1|_{\Lambda_1}),\qquad\text{ and }\qquad
    c_2 - \delta^1b_2 = \pi_*(c_2|_{\Lambda_1}).
\]
Hence
\begin{align*}
c_1 - c_2
    &= \pi_*(c_1|_{\Lambda_1}) + \delta^1b_1 - \pi_*(c_2|_{\Lambda_1}) - \delta^1b_2\\
    &= \delta^1b_1 - \delta^1b_2 + \pi_*(c_1|_{\Lambda_1} - c_2|_{\Lambda_1})\\
    &= \delta^1b_1 - \delta^1b_2 + \pi_*(\delta^1b).
\end{align*}
We have $\pi_*(b) \in C^1(\Lambda,A)$ and $\pi_*(\delta^1 b) = \delta^1(\pi_*b)$, so we
deduce that $c_1$ and $c_2$ are cohomologous.
\end{proof}

\begin{rmk}
Since twisted $C^*$-algebras do not ``see" a perturbation of a 2-cocycle by a coboundary
\cite[Proposition~5.6]{KumPasSim2}, Theorem~\ref{cor2.6} says that a twisted $k$-graph
$C^*$-algebra $C^*(\Lambda, c)$ associated to a cocycle $c$ on the $(k+1)$-graph
associated to covering sequence $(\Lambda_n, p_n)$ is determined by the covering sequence
and $c|_{\Lambda_1}$. One might expect that a similar statement applies to traces; we
addresses this in the next section.
\end{rmk}

\section{Bratteli diagrams of covering maps}\label{sec5}

We now consider a more complicated situation than in the preceding two sections: Rather
than a single covering system, we consider Bratteli diagrams of covering maps between $k$-graphs. Roughly
speaking this consists of a Bratteli diagram to which we associate a $k$-graph at each vertex and a covering map at each edge. In particular, each infinite path in
the diagram corresponds to a covering sequence of $k$-graphs to which we can apply the
results of the preceding two sections. We show how to construct a $(k+1)$-graph $\Lambda$
from a Bratteli diagram of covering maps, and how to view a full corner of a twisted
$C^*$-algebra $C^*(\Lambda, c)$ as a direct limit of direct sums of matrix algebras over
the algebras $C^*(\Lambda_v, c|_{\Lambda_v})$.

We take the convention that a Bratteli diagram is a 1-graph $E$ with a partition of $E^0$
into finite subsets $E^0=\bigsqcup_{n=1}^\infty E_n^0$ such that
\[
E^1 = \bigsqcup_{n=1}^\infty E_n^0 E^1 E_{n+1}^0.
\]
We let $E^*$ denote the set of finite paths in $E$. Following \cite{Bra} we insist that
\[
vE^1 \not= \emptyset\text{ for all $v$,}
    \quad\text{ and }\quad
E^1 v \not=\emptyset\text{ for all $v \in E^0 \setminus E_1^0$.}
\]
This implies that each $E_n^0$ is non-empty. In the usual convention for drawing directed
graphs the vertex set $E^0$ has levels $E_n^0$ arranged horizontally, and edges point
from right to left: each edge in $E$ points from some level $E_{n+1}^0$ to the level
$E_{n}^0$ immediately to its left.

We say that a Bratteli diagram $E$ is \emph{singly connected} if $|vE^1w|\leq 1$ for all
$v,w\in E^0$.

\begin{defn}
A \emph{Bratteli diagram of covering maps between $k$-graphs} consists of a singly
connected Bratteli diagram $E$, together with a collection $(\Lambda_v)_{v\in E^0}$ of
$k$-graphs and a collection $(p_e)_{e\in E^1}$ of covering maps $p_e\colon
\Lambda_{s(e)}\to \Lambda_{r(e)}$.
\end{defn}

A Bratteli diagram of covering maps is sketched below.
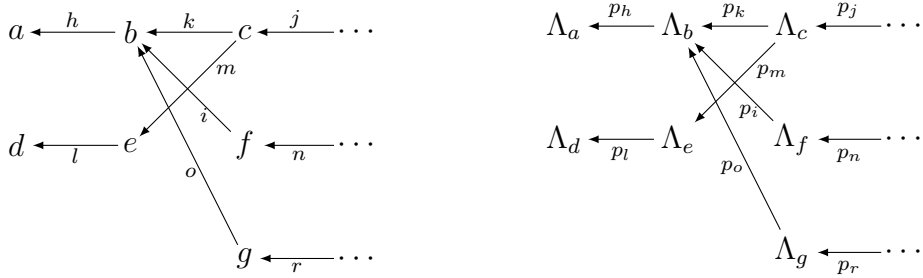
\[
\begin{tikzpicture}[scale=1.5]
    \node[inner sep=0.5pt, circle] (13) at (2,0) {$g$};
    \node[inner sep=0.5pt, circle] (14) at (3,0) {$\cdots$};
    \node[inner sep=0.5pt, circle] (21) at (0,1) {$d$};
    \node[inner sep=0.5pt, circle] (22) at (1,1) {$e$};
    \node[inner sep=0.5pt, circle] (23) at (2,1) {$f$};
    \node[inner sep=0.5pt, circle] (24) at (3,1) {$\cdots$};
    \node[inner sep=0.5pt, circle] (31) at (0,2) {$a$};
    \node[inner sep=0.5pt, circle] (32) at (1,2) {$b$};
    \node[inner sep=0.5pt, circle] (33) at (2,2) {$c$};
    \node[inner sep=0.5pt, circle] (34) at (3,2) {$\cdots$};
    \draw[-latex, black] (14)--(13) node[inner sep=0.5pt, circle, pos=0.5, below, black] {\tiny{$r$}};
    \draw[-latex, black] (13)--(32) node[inner sep=0.5pt, circle, pos=0.4, anchor=north east, black] {\tiny{$o$}};
    \draw[-latex, black] (32)--(31) node[inner sep=0.5pt, circle, pos=0.5, above, black] {\tiny{$h$}};
    \draw[-latex, black] (24)--(23) node[inner sep=0.5pt, circle, pos=0.5, below, black] {\tiny{$n$}};
    \draw[-latex, black] (34)--(33) node[inner sep=0.5pt, circle, pos=0.5, above, black] {\tiny{$j$}};
    \draw[-latex, black] (33)--(22) node[inner sep=0.5pt, circle, pos=0.2, anchor=north west, black] {\tiny{$m$}};
    \draw[-latex, black] (33)--(32) node[inner sep=0.5pt, circle, pos=0.5, above, black] {\tiny{$k$}};
    \draw[-latex, black] (22)--(21) node[inner sep=0.5pt, circle, pos=0.5, below, black] {\tiny{$l$}};
    \draw[-latex, black] (23)--(32) node[inner sep=0.5pt, circle, pos=0.3, below, black] {\tiny{$i$}};
\end{tikzpicture}\hskip5em
\begin{tikzpicture}[scale=1.5]
    \node[inner sep=0.5pt, circle] (13) at (2,0) {$\Lambda_g$};
    \node[inner sep=0.5pt, circle] (14) at (3,0) {$\cdots$};
    \node[inner sep=0.5pt, circle] (21) at (0,1) {$\Lambda_d$};
    \node[inner sep=0.5pt, circle] (22) at (1,1) {$\Lambda_e$};
    \node[inner sep=0.5pt, circle] (23) at (2,1) {$\Lambda_f$};
    \node[inner sep=0.5pt, circle] (24) at (3,1) {$\cdots$};
    \node[inner sep=0.5pt, circle] (31) at (0,2) {$\Lambda_a$};
    \node[inner sep=0.5pt, circle] (32) at (1,2) {$\Lambda_b$};
    \node[inner sep=0.5pt, circle] (33) at (2,2) {$\Lambda_c$};
    \node[inner sep=0.5pt, circle] (34) at (3,2) {$\cdots$};
    \draw[-latex, black] (14)--(13) node[inner sep=0.5pt, circle, pos=0.5, below, black] {\tiny{$p_r$}};
    \draw[-latex, black] (13)--(32) node[inner sep=0.5pt, circle, pos=0.4, anchor=north east, black] {\tiny{$p_o$}};
    \draw[-latex, black] (32)--(31) node[inner sep=0.5pt, circle, pos=0.5, above, black] {\tiny{$p_h$}};
    \draw[-latex, black] (24)--(23) node[inner sep=0.5pt, circle, pos=0.5, below, black] {\tiny{$p_n$}};
    \draw[-latex, black] (34)--(33) node[inner sep=0.5pt, circle, pos=0.5, above, black] {\tiny{$p_j$}};
    \draw[-latex, black] (33)--(22) node[inner sep=0.5pt, circle, pos=0.2, anchor=north west, black] {\tiny{$p_m$}};
    \draw[-latex, black] (33)--(32) node[inner sep=0.5pt, circle, pos=0.5, above, black] {\tiny{$p_k$}};
    \draw[-latex, black] (22)--(21) node[inner sep=0.5pt, circle, pos=0.5, below, black] {\tiny{$p_l$}};
    \draw[-latex, black] (23)--(32) node[inner sep=0.5pt, circle, pos=0.3, below, black] {\tiny{$p_i$}};
\end{tikzpicture}
\]

Given a Bratteli diagram of covering maps between $k$-graphs there exist a unique
$(k+1)$-graph, denoted $\Lambda_E$ (or simply $\Lambda$), together with injective
functors $\iota_v: \Lambda_v \to \Lambda$, $v\in E^0$, and a bijective map $e:
\bigsqcup_{v\in E^0\backslash E^0_1} \Lambda_v^0\times E^1v \to \Lambda^{e_{k+1}}$, such
that
\begin{enumerate}
\item $d(\iota_v(\lambda))=(d(\lambda),0)$ for $\lambda\in \Lambda_v$,
\item $\iota_v(\Lambda_v) \cap \iota_w(\Lambda_w) = \emptyset$ for $v\neq w \in E^0$,
\item $\bigsqcup_{v\in E^0} \iota_v(\Lambda_v) = \{\lambda \in \Lambda :
    d(\lambda)_{k+1}=0\}$,
\item $s(e(w,f))=\iota_v(w)$ and $r(e(w,f))=\iota_{r(f)}(p_f(w))$ for $v\in
    E^0\backslash E^0_1$ and $(w,f)\in \Lambda_v^0\times E^1v$, and
\item $e(r(\lambda),f)\iota_v(\lambda) = \iota_{r(f)}(p_f(\lambda))e(s(\lambda),f)$
    for $v\in E^0\backslash E^0_1$ and $(\lambda,f)\in \Lambda_{v}^0\times E^1v$.
\end{enumerate}
The construction of the $(k+1)$-graph $\Lambda$ is taken from \cite{KumPasSim}, with the
exception that we describe the underlying data as a Bratteli diagram rather than a
sequence of $\{0,1\}$-valued matrices. We will view each $\Lambda_v$ as a subset of
$\Lambda$.

\begin{defn}[c.f.~Definition~\ref{def2.2}]
\label{def5.5} Let $\Lambda=\Lambda_E$ be the $(k+1)$-graph associated to a Bratteli
diagram of covering maps between $k$-graphs, and let $c_v\in Z^2(\Lambda_v, A)$ for each
$v\in E^0$. The collection of 2-cocycles $(c_{v})$ is called \emph{compatible} if there
exists a 2-cocycle $c\in Z^2(\Lambda, A)$ such that $c|_{\Lambda_{v}} = c_v$ for all
$v\in E^0$.
\end{defn}

Definition~\ref{def5.5} shows how to build a twisted $(k+1)$-graph $C^*$-algebra from a
Bratteli diagram of covering maps. We will exhibit this $C^*$-algebra as an inductive
limit. This involves considering homomorphisms between twisted $k$-graph $C^*$-algebras
associated to subgraphs of the ambient $(k+1)$-graph $\Lambda_E$ associated to the
Bratteli diagram of covering maps. In keeping with this, we use the same symbol $s$ to
denote the generating twisted Cuntz-Krieger families $s\colon\lambda \mapsto s_\lambda$ of
the $C^*$-algebras of the different subgraphs. It will be clear from context which
$k$-graph we are working in at any given time.

\begin{lemma}
\label{lem5.5} Let $\Lambda=\Lambda_E$ be the $(k+1)$-graph associated to a Bratteli
diagram of covering maps between row finite locally convex $k$-graphs, together with a
compatible collection $(c_v)$ of 2-cocycles. For each $e\in E^1$ there exists an
embedding $\iota_e\colon C^*(\Lambda_{r(e)}, c_{r(e)})\to C^*(\Lambda_{s(e)}, c_{s(e)})$
such that
\[
\iota_e(s_\lambda) = \sum_{p_e(\mu) = \lambda} s_\mu \quad\text{ for all $\lambda\in \Lambda_{r(e)}$.}
\]
\end{lemma}
\begin{proof}
We follow the argument of \cite[Remark~3.5.(2)]{KumPasSim} (which applies in the
situation where $c_{r(e)} = c_{s(e)} \equiv 1$). The argument goes through mutatis
mutandis when the cocycles are nontrivial.
\end{proof}

With slight abuse of notation, in the situation of Lemma~\ref{lem5.5}, given $n\geq 1$,
we also write $\iota_e$ for the induced map $\iota^{(n)}_{e}\colon
M_n(C^*(\Lambda_{r(e)}, c_{r(e)}))\to M_n(C^*(\Lambda_{s(e)}, c_{s(e)}))$ given by
\[
\left(\begin{array}{ccc}
	a_{11} & \hdots & a_{1n} \\
	\vdots & \ddots & \vdots \\
		a_{n1} & \hdots & a_{nn}
\end{array}\right)
\mapsto
\left(\begin{array}{ccc}
	\iota_e(a_{11}) & \hdots & \iota_e(a_{1n}) \\
	\vdots & \ddots & \vdots \\
		\iota_e(a_{n1}) & \hdots & \iota_e(a_{nn})
\end{array}\right).
\]

\begin{thm}\label{thm5.6}
Let $\Lambda=\Lambda_E$ be the $(k+1)$-graph associated to a Bratteli diagram of covering
maps between row finite locally convex $k$-graphs, together with a compatible collection
$(c_v)$ of 2-cocycles. Let $c\in Z^2(\Lambda,\TT)$ be a 2-cocycle such that
$c|_{\Lambda_{v}} = c_v$. The projection $P_0 := \sum_{v\in E_1^0,\, w\in \Lambda_v^0}
s_w$ is full in $C^*(\Lambda, c)$. For $n \in \NN$, let $A_n := \clsp\{s_\mu s^*_\nu :
\mu,\nu \in \Lambda, d(\mu)_{k+1} = d(\nu)_{k+1} = n\}$. Then each $A_n \subseteq
A_{n+1}$, and
\[\textstyle
P_0 C^*(\Lambda,c) P_0
    = \overline{\bigcup_n A_n}.
\smallskip\]
For each $\alpha=\alpha_1 \dots \alpha_n \in E^0_1 E^* E_n^0$, let $F(\alpha) :=
\{\mu=\mu_1 \dots \mu_n \in \Lambda^{ne_{k+1}} : r(\mu_i) \in \Lambda^0_{r(\alpha_i)}$
for $i \le n$, and $s(\mu_n) \in \Lambda^0_{s(\alpha_n)}\}$. Then each $T_\alpha :=
\sum_{\mu \in F(\alpha)} s_\mu$ is a partial isometry, and there is an isomorphism
$\omega_n\colon\bigoplus_{v\in E^0_n} M_{E^0_1E^*v}(C^*(\Lambda_v, c_v)) \to A_n$ such that
\[
\omega_n\Big(\big((a_{\alpha,\beta})_{\alpha,\beta \in E^0_1E^*v}\big)_v\Big)
    =\sum_v \sum_{\alpha,\beta \in E^0_1 E^* v} T_\alpha a_{\alpha,\beta} T^*_\beta.
\]
We have $\omega_{n+1} \circ (\operatorname{diag}_{{e\in E^1w}} \iota_e) = \omega_n$, and so
\[
P_0 C^*(\Lambda, c) P_0
    \cong \varinjlim \left(\bigoplus_{v\in E^0_n} M_{E^0_1E^*v}(C^*(\Lambda_v, c_v)),\,
        \sum_{v\in E^0_n}a_v \mapsto \sum_{w\in E^0_{n+1}}
            \operatorname{diag}_{{e\in E^1w}} \left(\iota_e(a_{r(e)})\right)\right).
\]
\end{thm}
\begin{proof}
We follow the proof of \cite[Theorem~3.8]{KumPasSim}, using Lemma~\ref{lem5.5} in place
of \cite[Remark~3.5.2]{KumPasSim}.
\end{proof}

\section{Rank-3 Bratteli diagrams and their \texorpdfstring{$C^*$}{C*}-algebras}\label{sec:rank3BD}
We are now ready to prove our main results. We consider a special class of $3$-graphs,
which we call ``rank-3 Bratteli diagrams" (see Definition~\ref{def5.6}). We will compute
the $K$-theory of the twisted $C^*$-algebras of these $3$-graphs associated to twists by
irrational angles using the inductive-limit decomposition just described and the
well-known formula for the ordered $K$-theory of the irrational rotation algebras. We
will deduce from Elliott's theorem that these $C^*$-algebras are classified by $K$-theory
whenever they are simple, and deduce that such $C^*$-algebras, when simple, can all be
realised by rank-2 Bratteli diagrams as in \cite{PasRaeRorSim}.

Recall that a Bratteli diagram is singly connected if $|vE^1w|\leq 1$ for all $v,w\in
E^0$. Our rank-3 Bratteli diagrams will be constructed from singly connected Bratteli
diagrams of coverings of $2$-graphs where the individual $2$-graphs are rank-2 simple
cycles. The lengths of the cycles are encoded by an additional piece of information: a
weight map on the underlying Bratteli diagram.

\begin{defn}
Let $E$ be a Bratteli diagram. A \emph{weight map} on $E$ is a function $w \colon E^0\to
\NN\backslash \{0\}$ such that $w(r(e))$ divides $w(s(e))$ for all $e\in E^1$. A
\emph{weighted Bratteli diagram} is a Bratteli diagram $E$ together with a weight map.
\end{defn}

An example of the first few levels of a singly connected weighted Bratteli diagram is
sketched below (the weight map $w$ is identified by labelling the vertices):
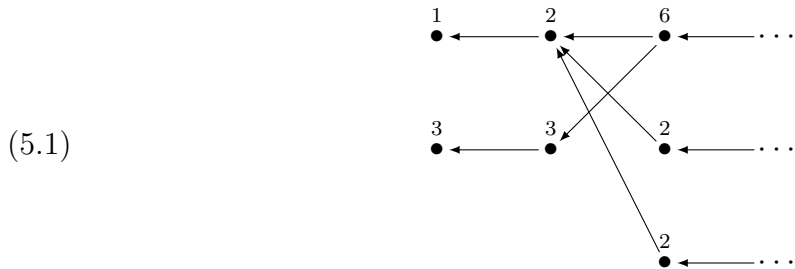
\begin{equation}\label{eq:wbd}
\parbox{5cm}{
\begin{tikzpicture}[scale=1.5]
    \node[inner sep=0.5pt, circle] (13) at (2,0) {\small$\bullet$};
    \node[anchor=south, inner sep=1pt] at (13.north) {\tiny$2$};
    \node[inner sep=0.5pt, circle] (14) at (3,0) {$\cdots$};
    \node[inner sep=0.5pt, circle] (21) at (0,1) {\small$\bullet$};
    \node[anchor=south, inner sep=1pt] at (21.north) {\tiny$3$};
    \node[inner sep=0.5pt, circle] (22) at (1,1) {\small$\bullet$};
    \node[anchor=south, inner sep=1pt] at (22.north) {\tiny$3$};
    \node[inner sep=0.5pt, circle] (23) at (2,1) {\small$\bullet$};
    \node[anchor=south, inner sep=1pt] at (23.north) {\tiny$2$};
    \node[inner sep=0.5pt, circle] (24) at (3,1) {$\cdots$};
    \node[inner sep=0.5pt, circle] (31) at (0,2) {\small$\bullet$};
    \node[anchor=south, inner sep=1pt] at (31.north) {\tiny$1$};
    \node[inner sep=0.5pt, circle] (32) at (1,2) {\small$\bullet$};
    \node[anchor=south, inner sep=1pt] at (32.north) {\tiny$2$};
    \node[inner sep=0.5pt, circle] (33) at (2,2) {\small$\bullet$};
    \node[anchor=south, inner sep=1pt] at (33.north) {\tiny$6$};
    \node[inner sep=0.5pt, circle] (34) at (3,2) {$\cdots$};
    \draw[-latex, black] (14)--(13);
    \draw[-latex, black] (13)--(32);
    \draw[-latex, black] (32)--(31);
    \draw[-latex, black] (24)--(23);
    \draw[-latex, black] (34)--(33);
    \draw[-latex, black] (33)--(22);
    \draw[-latex, black] (33)--(32);
    \draw[-latex, black] (22)--(21);
    \draw[-latex, black] (23)--(32);
\end{tikzpicture}}
\end{equation}

To construct our rank-3 Bratteli diagrams, we need to recall the skew-product
construction for $k$-graphs. Following \cite{KumPas}, fix a $k$-graph $\Lambda$ and a functor $\eta:\Lambda\to G$ into
a countable group $G$. The \emph{skew product graph}, denoted $\Lambda\times_\eta G$, is
the $k$-graph with morphisms $\Lambda\times G$, source, range and degree maps given by
\[
    r(\lambda,g)=(r(\lambda),g), \quad s(\lambda,g)=(s(\lambda),g\eta(\lambda)),
        \quad d(\lambda,g)=d(\lambda),
\]
and composition given by $(\lambda,g)(\mu,h)=(\lambda\mu,g)$ whenever
$s(\lambda,g)=r(\mu,h)$.

\begin{defn}\label{def5.6}
Let $E$ be a singly connected weighted Bratteli diagram. For $v\in E^0$, let $a_v, b_v$
be the blue and red (respectively) edges in a copy $T_2^v$ of $T_2$. For each $v$, let $1
: T_2^v \to \ZZ/w(v)\ZZ$ be the functor such that $1(a_v) = 1(b_v) = 1$, the generator of
$\ZZ/w(v)\ZZ$. The \emph{rank-3 Bratteli diagram $\Lambda_E$} (or simply $\Lambda$)
associated to $E$ is the unique 3-graph arising from the Bratteli diagram of covering
maps given by
\begin{align*}
\Lambda_v &= T_2^v \times_1  \ZZ / w(v)\ZZ, &v&\in E^0,\\
p_f(a_{s(f)}^sb_{s(f)}^t,m) &= (a_{r(f)}^sb_{r(f)}^t,m  \text{ mod } w(r(f)), &s,t&\in \NN, f\in E^1,m\in \ZZ/w(s(f))\ZZ.
\end{align*}
To keep notation compact we write $\{(v,m): m=0,\dots, w(v)-1\}$ for the vertices of
$\Lambda_v$.

\end{defn}
Figure~\ref{figpic6} illustrates the portion of the skeleton of a rank-3 Bratteli
diagram corresponding to the portion of a weighted Bratteli in~\eqref{eq:wbd}.
\begin{figure}[ht]
\begin{center}
\begin{tikzpicture}[xscale=2, yscale=1.25]
    \node[inner sep=0.5pt, circle] (40) at (4,1) {\tiny${(h,1)}$};
    \node[inner sep=0.5pt, circle] (50) at (6,1) {$\cdots$};
    \node[inner sep=0.5pt, circle] (41) at (4,2) {\tiny${(h,0)}$};
    \node[inner sep=0.5pt, circle] (51) at (6,2) {$\cdots$};
    \node[inner sep=0.5pt, circle] (04) at (0,4) {\tiny${(d,2)}$};
    \node[inner sep=0.5pt, circle] (24) at (2,4) {\tiny${(e,2)}$};
    \node[inner sep=0.5pt, circle] (05) at (0,5) {\tiny${(d,1)}$};
    \node[inner sep=0.5pt, circle] (25) at (2,5) {\tiny${(e,1)}$};
    \node[inner sep=0.5pt, circle] (42) at (4,3) {\tiny${(f,1)}$};
    \node[inner sep=0.5pt, circle] (52) at (6,3) {$\cdots$};
    \node[inner sep=0.5pt, circle] (06) at (0,6) {\tiny${(d,0)}$};
    \node[inner sep=0.5pt, circle] (26) at (2,6) {\tiny${(e,0)}$};
    \node[inner sep=0.5pt, circle] (43) at (4,4) {\tiny${(f,0)}$};
    \node[inner sep=0.5pt, circle] (53) at (6,4) {$\cdots$};
    \node[inner sep=0.5pt, circle] (45) at (4,5) {\tiny${(c,5)}$};
    \node[inner sep=0.5pt, circle] (55) at (6,5) {$\cdots$};
    \node[inner sep=0.5pt, circle] (46) at (4,6) {\tiny${(c,4)}$};
    \node[inner sep=0.5pt, circle] (56) at (6,6) {$\cdots$};
    \node[inner sep=0.5pt, circle] (47) at (4,7) {\tiny${(c,3)}$};
    \node[inner sep=0.5pt, circle] (57) at (6,7) {$\cdots$};
    \node[inner sep=0.5pt, circle] (48) at (4,8) {\tiny${(c,2)}$};
    \node[inner sep=0.5pt, circle] (58) at (6,8) {$\cdots$};
    \node[inner sep=0.5pt, circle] (27) at (2,7) {\tiny${(b,1)}$};
    \node[inner sep=0.5pt, circle] (49) at (4,9) {\tiny${(c,1)}$};
    \node[inner sep=0.5pt, circle] (59) at (6,9) {$\cdots$};
    \node[inner sep=0.5pt, circle] (07) at (0,7.5) {\tiny${(a,0)}$};
    \node[inner sep=0.5pt, circle] (28) at (2,8) {\tiny${(b,0)}$};
    \node[inner sep=0.5pt, circle] (410) at (4,10) {\tiny${(c,0)}$};
    \node[inner sep=0.5pt, circle] (510) at (6,10) {$\cdots$};
    \draw[-latex, red, dashed] (27) edge[out=180,in=180] (28);
    \draw[-latex, blue] (27) edge[out=160,in=200] (28);
    \draw[-latex, red, dashed] (28) edge[out=0,in=0] (27);
    \draw[-latex, blue] (28) edge[out=340,in=20] (27);
    \draw[-latex, red, dashed] (46) edge[out=180,in=180] (47);
    \draw[-latex, blue] (46) edge[out=160,in=200] (47);
    \draw[-latex, red, dashed] (410) edge[out=340,in=20] (45); 
    \draw[-latex, blue] (410) edge[out=330,in=30] (45); 
    \draw[-latex, red, dashed] (45) edge[out=180,in=180] (46);
    \draw[-latex, blue] (45) edge[out=160,in=200] (46);
    \draw[-latex, red, dashed] (46) edge[out=180,in=180] (47);
    \draw[-latex, blue] (46) edge[out=160,in=200] (47);
    \draw[-latex, red, dashed] (47) edge[out=180,in=180] (48);
    \draw[-latex, blue] (47) edge[out=160,in=200] (48);
    \draw[-latex, red, dashed] (48) edge[out=180,in=180] (49);
    \draw[-latex, blue] (48) edge[out=160,in=200] (49);
    \draw[-latex, red, dashed] (49) edge[out=180,in=180] (410);
    \draw[-latex, blue] (49) edge[out=160,in=200] (410);
    \draw[-latex, red, dashed] (04) edge[out=180,in=180] (05);
    \draw[-latex, blue] (04) edge[out=160,in=200] (05);
    \draw[-latex, red, dashed] (05) edge[out=180,in=180] (06);
    \draw[-latex, blue] (05) edge[out=160,in=200] (06);
    \draw[-latex, red, dashed] (06) edge[out=0,in=0] (04);
    \draw[-latex, blue] (06) edge[out=340,in=20] (04);
    \draw[-latex, red, dashed] (24) edge[out=180,in=180] (25);
    \draw[-latex, blue] (24) edge[out=160,in=200] (25);
    \draw[-latex, red, dashed] (25) edge[out=180,in=180] (26);
    \draw[-latex, blue] (25) edge[out=160,in=200] (26);
    \draw[-latex, red, dashed] (26) edge[out=0,in=0] (24);
    \draw[-latex, blue] (26) edge[out=340,in=20] (24);
    \draw[-latex, red, dashed] (42) edge[out=180,in=180] (43);
    \draw[-latex, blue] (42) edge[out=160,in=200] (43);
    \draw[-latex, red, dashed] (43) edge[out=0,in=0] (42);
    \draw[-latex, blue] (43) edge[out=340,in=20] (42);
    \draw[-latex, red, dashed] (40) edge[out=180,in=180] (41);
    \draw[-latex, blue] (40) edge[out=160,in=200] (41);
    \draw[-latex, red, dashed] (41) edge[out=0,in=0] (40);
    \draw[-latex, blue] (41) edge[out=340,in=20] (40);
	\path[->,every loop/.style={looseness=10}] (07)
	         edge  [in=70,out=110,loop, blue] ();
	\path[->,every loop/.style={looseness=10}] (07)
			 edge  [in=60,out=120,loop, red, dashed] ();
    \draw[-latex, green!50!black, dotted, very thick] (40)--(27);
    \draw[-latex, green!50!black, dotted, very thick] (41)--(28);
    \draw[-latex, green!50!black, dotted, very thick] (42)--(27);
    \draw[-latex, green!50!black, dotted, very thick] (43)--(28);
    \draw[-latex, green!50!black, dotted, very thick] (46)--(26);
    \draw[-latex, green!50!black, dotted, very thick] (47)--(27);
    \draw[-latex, green!50!black, dotted, very thick] (27)--(07);
    \draw[-latex, green!50!black, dotted, very thick] (28)--(07);
    \draw[-latex, green!50!black, dotted, very thick] (45)--(27);
    \draw[-latex, green!50!black, dotted, very thick] (46)--(28);
    \draw[-latex, green!50!black, dotted, very thick] (47)--(27);
    \draw[-latex, green!50!black, dotted, very thick] (48)--(28);
    \draw[-latex, green!50!black, dotted, very thick] (49)--(27);
    \draw[-latex, green!50!black, dotted, very thick] (410)--(28);
    \draw[-latex, green!50!black, dotted, very thick] (24)--(04);
    \draw[-latex, green!50!black, dotted, very thick] (25)--(05);
    \draw[-latex, green!50!black, dotted, very thick] (26)--(06);
    \draw[-latex, green!50!black, dotted, very thick] (410)--(26);
    \draw[-latex, green!50!black, dotted, very thick] (49)--(25);
    \draw[-latex, green!50!black, dotted, very thick] (48)--(24);
    \draw[-latex, green!50!black, dotted, very thick] (47)--(26);
    \draw[-latex, green!50!black, dotted, very thick] (46)--(25);
    \draw[-latex, green!50!black, dotted, very thick] (45)--(24);
    \draw[-latex, green!50!black, dotted, very thick] (510)--(410);
    \draw[-latex, green!50!black, dotted, very thick] (59)--(49);
    \draw[-latex, green!50!black, dotted, very thick] (58)--(48);
    \draw[-latex, green!50!black, dotted, very thick] (57)--(47);
    \draw[-latex, green!50!black, dotted, very thick] (56)--(46);
    \draw[-latex, green!50!black, dotted, very thick] (55)--(45);
    \draw[-latex, green!50!black, dotted, very thick] (53)--(43);
    \draw[-latex, green!50!black, dotted, very thick] (52)--(42);
    \draw[-latex, green!50!black, dotted, very thick] (51)--(41);
    \draw[-latex, green!50!black, dotted, very thick] (50)--(40);
\end{tikzpicture}
\caption{} \label{figpic6}
\end{center}
\end{figure}

\begin{rmk}
Our definition of rank-3 Bratteli diagram relates to the rank-2 Bratteli diagrams
introduced by Pask, Raeburn, R{\o}rdam and Sims. Both constructions are based on Bratteli
diagrams as initial data, with the difference that here we construct 3-graphs rather than
2-graphs. See~\cite{PasRaeRorSim} for the details.
\end{rmk}

Given a $\TT$-valued 2-cocycle $c$ on $\ZZ^k$ and a $k$-graph $\Lambda$, we obtain a
$2$-cocycle $\dstar{}c$ on $\Lambda$ by $(\dstar{}c)(\mu,\nu) = c(d(\mu), d(\nu))$. An
example of a $\TT$-valued 2-cocycle on $\ZZ^k$ is the map
\begin{equation}\label{ck}
c^k_\theta(m,n) = e^{2\pi  \theta m_2 n_1}.
\end{equation}
The 2-cocycle (\ref{ck}) on $\ZZ^2$ will be denoted $c_\theta$. We let $A_\theta$ denote
the \emph{rotation $C^*$-algebra} corresponding to the angle $\theta\in \RR$ (see
\cite[Example~7.7]{KumPasSim3}).

Our main theorem describes the ordered $K$-theory of the twisted $C^*$-algebras of rank-3
Bratteli diagrams corresponding to irrational $\theta$. We state the result now, but the
proof will require some more preliminary work.

To state the theorem, we take the convention that given a direct sum $G = \bigoplus_i
G_i$ of groups and $g \in G$, we write $g\delta_i$ for the image of $g$ in the
$i$\textsuperscript{th} direct summand of $G$.

\begin{thm}
\label{cor5.8} Let $\Lambda=\Lambda_E$ be the rank-3 Bratteli diagram associated to a
singly connected weighted Bratteli diagram $E$, and take $\theta \in \RR \backslash \QQ$.
Let $c^3_\theta \in Z^2(\ZZ^3, \TT)$ be as in~\eqref{ck}, and let
$c=\dstar{}c^3_\theta\in Z^2(\Lambda,\TT)$. For $n \in \NN$, define $A_n: \oplus_{v \in
E_n^0} (\frac{1}{w(v)}\ZZ + \theta \ZZ)  \to \oplus_{u \in E^0_{n+1}} (\frac{1}{w(u)}\ZZ
+ \theta \ZZ)$ and $B_n\colon\oplus_{v \in E_n^0} \ZZ^2 \to \oplus_{u \in E_{n+1}^0}
\ZZ^2$ by
\begin{align*}
A_n\Big(\frac{1}{w(v)}p+\theta q\Big)\delta_v
    &=\sum_{e\in vE^1}\Big(\frac{1}{w(v)}p+\theta q\Big)\delta_{s(e)}\qquad\text{ and}\\
B_n(p, q)\delta_v
    &=\sum_{e\in vE^1}\Big(p+\Big(1-\frac{w(s(e))}{w(v)}\Big)q,
        \, \frac{w(s(e))}{w(v)}q\Big)\delta_{s(e)}.
\end{align*}
Endow $\varinjlim \Big(\bigoplus_{v\in E_n^0} (\frac{1}{w(v)}\ZZ+\theta\ZZ), A_n\Big)$ with
the positive cone and order inherited from the approximating subgroups. Then there are an
order isomorphism
\begin{align*}
h_0\colon K_0(C^*(\Lambda,c)) &\to \varinjlim
    \Big(\bigoplus_{v\in E_n^0} \Big(\frac{1}{w(v)}\ZZ+\theta\ZZ\Big), A_n\Big)
\intertext{and an isomorphism}
h_1\colon K_1(C^*(\Lambda,c)) &\to \varinjlim
            \Big(\bigoplus_{v\in E_n^0} \ZZ^2, B_n\Big)
\end{align*}
with the following properties: Let $v \in E_n^0$, let $i < w(v)$, let $\mu_i$ be the
unique element of $(v,i)\Lambda_v^{(w(v), 0)}$ and let $\nu_i$ be the unique element of
$(v,i)\Lambda_v^{(w(v) - 1, 1)}$. Then
\begin{align*}
\textstyle h_0\big(\big[s_{(v,i)}\big]\big)
    &=\textstyle \frac{1}{w(v)} \delta_v \in \bigoplus_{u\in E_n^0} (\frac{1}{w(u)}\ZZ+\theta\ZZ);\\
\textstyle h_1\big(\big[s_{\mu_i} + \sum_{j \not= i} s_{(v,j)}\big]\big)
    &=\textstyle (1,0)\delta_v \in \bigoplus_{u \in E_n^0} \ZZ^2;\text{ and}\\
\textstyle h_1\big(\big[s_{\nu_i} + \sum_{j \not= i} s_{(v,j)}\big]\big)
    &=\textstyle (0,1)\delta_v \in \bigoplus_{u \in E_n^0} \ZZ^2.
\end{align*}
There is an isomorphism $\theta\colon K_1(C^*(\Lambda, c)) \to K_0(C^*(\Lambda), c)$ such
that $h_0 \circ \theta \circ h_1^{-1}\big((a,b)\delta_v\big) = \frac{b}{w(v)} + (a +
b)\theta$ for all $v \in E^0$ and $(a,b) \in \ZZ^2$.
\end{thm}

\begin{rmk}\label{rmk:matrices of maps}
If we regard the maps $A_n\colon\oplus_{v \in E_n^0} (\frac{1}{w(v)}\ZZ + \theta \ZZ)  \to
\oplus_{u \in E^0_{n+1}} (\frac{1}{w(u)}\ZZ + \theta \ZZ)$ as $E^0_n \times E^0_{n+1}$
matrices of homomorphisms $A_n(v,u)\colon\frac{1}{w(v)}\ZZ + \theta \ZZ  \to
\frac{1}{w(u)}\ZZ + \theta \ZZ$, then each $A_n(v,w)$ is the inclusion map
$\frac{1}{w(v)}\ZZ + \theta \ZZ  \subseteq \frac{1}{w(u)}\ZZ + \theta \ZZ$. Likewise, if
we think of each $B_n$ as an $E^0_n \times E^0_{n+1}$ matrix of homomorphisms $B_n(v,u) :
\ZZ^2 \to \ZZ^2$, then writing $l(v,u) := w(u)/w(v)$, each $B_n(v,u)$ is implemented by
multiplication by the matrix $\left(\begin{smallmatrix} 1 & 1 - l(v,u) \\ 0 &
l(v,u)\end{smallmatrix}\right)$.
\end{rmk}

\begin{rmk}
In the preceding theorem, given $v \in E_n^0$, it is not so easy to specify explicitly
the projection in $C^*(\Lambda, c)$ which maps to $\theta \delta_v \in \bigoplus_{u\in
E_n^0} (\frac{1}{w(u)}\ZZ+\theta\ZZ)$. However, the description of the connecting maps in
$K_0$ in Lemma~\ref{lem5.12} yields the following description: Lemma~\ref{lem5.7}
describes an isomorphism $\phi\colon A_{w(v)\theta}\otimes M_{w(v)}(\CC) \cong
C^*(\Lambda_v,c)$, and if $p_\theta$ denotes the Rieffel projection in $A_{w(v)\theta}$,
then $h_0$ carries the $K_0$-class of $\phi(p_\theta \oplus 0_{w(v)-1})$ to
$\theta\delta_v$.
\end{rmk}

\begin{rmk}\label{rmk:easycocycles}
Theorem~\ref{cor5.8} applies only when $c=\dstar{}c^3_\theta\in Z^2(\Lambda,\TT)$ for
some $\theta$. But Theorem~\ref{cor2.6} suggests that this is a fairly mild hypothesis.
To make this precise, first suppose that $v,w \in E_1^0$ have the property that there
exist $\alpha \in vE^*$ and $\beta \in wE^*$ with $s(\alpha) = s(\beta)$. By choosing any
infinite path $x$ in $E$ with range $s(\alpha)$ we can pick out covering systems
corresponding to $\alpha x$ and to $\beta x$, and then Theorem~\ref{cor2.6} implies that
$c_v$ and $c_w$ are cohomologous. Now consider any connected component $C$ of $E$. An
induction using what we have just showed proves that the $c_v$ corresponding to vertices
$v$ in $C$ are all cohomologous. So, decomposing, $E$ into connected components, we see
that $C^*(\Lambda_E, c)$ is a direct sum of subalgebras $C^*(\Lambda_{C}, c)$ in which $v
\mapsto c_v$ is constant up to cohomology.

Each $\Lambda_v$ is the quotient of the $2$-graph $\Delta_2$ (see
\cite[Examples~2.2(5)]{KumPasSim3}) by the canonical action of $\{(m,n) : m+n \in
w(v)\ZZ\} \le \ZZ^2$, and so \cite[Theorem~4.9]{KumPasSim3} shows that $H_2(\Lambda_v,
\TT) \cong \TT$; in particular $\{[\dstar{}c_\theta] : \theta \in [0,2\pi)\}$ is all of
$H_2(\Lambda_v, \TT)$. So for each connected component $C$ of $E$, there is some $\theta$
such that $c_v \sim \dstar{}c_\theta$ for each vertex $v$ in $C$.
\end{rmk}

The first step to proving Theorem~\ref{cor5.8} is to describe the building blocks in the
direct-limit decomposition of Theorem~\ref{thm5.6} for a rank-3 Bratteli diagram. Recall
that if $E$ is a singly connected weighted Bratteli diagram and $n = w(v)$, then
$\Lambda_v = T_2^v \times_1  \ZZ / n\ZZ$ as illustrated below:
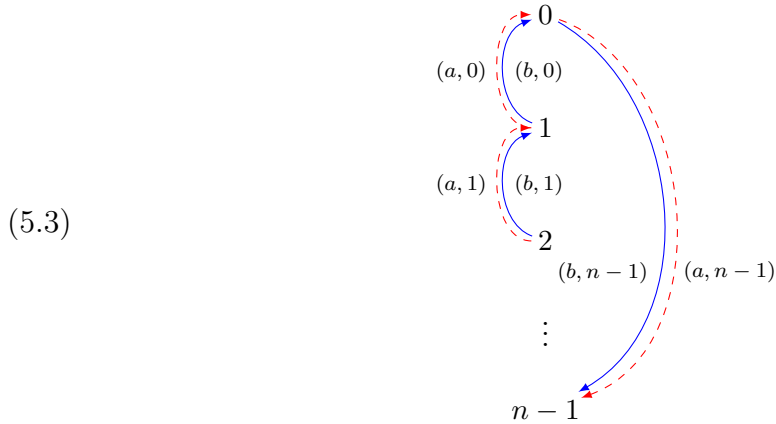
\begin{equation}\label{eq:doublecycle}
\parbox{5cm}{
\begin{tikzpicture}[scale=1.5]
    \node[inner sep=0.5pt, circle] (1) at (0,3) {\small$0$};
    \node[inner sep=0.5pt, circle] (2) at (0,2) {\small$1$};
    \node[inner sep=0.5pt, circle] (3) at (0,1) {\small$2$};
    \node[inner sep=0.5pt, circle] (4) at (0,-0.5) {\small$n-1$};
    \draw[-latex, red, dashed] (2) edge[out=180,in=180] node[pos=0.5, left, black] {\tiny{$(a,0)$}} (1);
    \draw[-latex, blue] (2) edge[out=160,in=200] node[pos=0.5, right, black] {\tiny{$(b,0)$}} (1);
    \draw[-latex, red, dashed] (3) edge[out=180,in=180] node[pos=0.5, left, black] {\tiny{$(a,1)$}}(2);
    \draw[-latex, blue] (3) edge[out=160,in=200] node[pos=0.5, right, black] {\tiny{$(b,1)$}} (2);
    \node at (0,0.25) {$\vdots$};
    \draw[-latex, red, dashed] (1) edge[out=340,in=20] node[pos=0.63, right, black] {\tiny{$(a,n-1)$}} (4);
    \draw[-latex, blue] (1) edge[out=330,in=30] node[pos=0.65, left, black] {\tiny{$(b,n-1)$}} (4);
\end{tikzpicture}}
\end{equation}
When $n=1$ the $C^*$-algebra $C^*(\Lambda_v,\dstar{}c^3_\theta)$ is isomorphic to the irrational
rotation algebra $A_\theta$ (see \cite{KumPasSim3}). We prove that, in general,
$C^*(\Lambda_v, \dstar{}c^3_\theta)$ is isomorphic to $M_n(A_{n\theta})$.

\begin{lemma}\label{lem5.7}
Let $\Lambda$ be the 2-graph $T_2 \times_1  \ZZ / n\ZZ$ of~\eqref{eq:doublecycle}. Let $c
= \dstar{}c_\theta$ for $\theta\in \RR$. Let $u,v$ denote the generators for $A_{n\theta}$,
and $(\zeta_{i,j})$ the standard matrix units for $M_{n}(\CC)$. Let ${\mu_{i}}$ (resp.
${\nu_{i}}$) denote the unique element in $\Lambda$ of degree $(n,0)$ (resp. $(n-1,1)$)
with source and range $(v,i)$, and let $\alpha_0, \dots, \alpha_{n-1}$ be any elements in
$\TT$ such that $e^{2\pi i\theta}\alpha_{i-1}=\alpha_i$. Then there is an isomorphism
\[
    \phi\colon A_{n\theta}\otimes M_{n}(\CC)\cong C^*(\Lambda,c)
\]
such that
\[
\phi(u\otimes 1_{n}) = \sum_{i=0}^{n-1} s_{\mu_{i}}, \quad
   \phi(v\otimes 1_{n})=\sum_{i=0}^{n-1} \alpha_{i}s_{\nu_{i}},\quad\text{and}\quad
   \phi(1\otimes \zeta_{j,j+1}) = s_{(a,j-1)}.
\]
\end{lemma}	
\begin{proof}
Define elements $U, V$ and $e_{j, j+1}$, $j < n$ of $C^*(\Lambda,c)$ by
\[
U = \sum_{i=0}^{n-1} s_{\mu_{i}}, \quad
    V = \sum_{i=0}^{n-1} \alpha_{i}s_{\nu_{i}}
    \quad\text{ and } e_{j,j+1}=s_{(a,j-1)}
\]
The set $\{e_{j,j+1}: j=1,\dots,n-1\}$ generates a system of matrix units
$(e_{i,j})_{i,j=1,\dots,n}$. Each $e_{i,j}$ is non-zero by \cite[Theorem
3.15]{RaeSimYee}. Straightforward calculations show that $U$ and $V$ are both unitaries,
and that these unitaries commute with the $e_{j,j+1}$ and their adjoints.

We claim that $VU=e^{2\pi i n\theta}UV$. Since $s(\mu_i) = r(\nu_j)$ only for $i = j$, we
have
\begin{align*}
UV &= \Big(\sum_{i=0}^{n-1} s_{\mu_{i}}\Big)\Big(\sum_{i=0}^{n-1} \alpha_{i}s_{\nu_{i}}\Big)
    = \sum_{i=0}^{n-1} \alpha_{i} c(\mu_{i},\nu_{i}) s_{\mu_{i}\nu_{i}}\quad\text{ and similarly}\\
VU &= \Big(\sum_{i=0}^{n-1} \alpha_{i}s_{\nu_{i}}\Big)\Big(\sum_{i=0}^{n-1} s_{\mu_{i}}\Big)
    =\sum_{i=0}^{n-1} \alpha_{i} c(\nu_{i},\mu_{i}) s_{\nu_{i}\mu_{i}}.
\end{align*}
Since $c(\mu_{i},\nu_{i}) = c_\theta\big((n, 0), (n-1,1)\big)=1$ and $c(\nu_{i},\mu_{i})
= c_\theta\big((n-1,1),(n, 0)\big) = e^{2\pi i n\theta}$, we obtain $VU = e^{2\pi i
n\theta}UV$ as claimed.

By definition the elements $U, V, e_{i,j}$ (and $\{s_\lambda : \lambda\in\Lambda\}$)
belong to the algebra generated by the elements $\{s_{(a,0)}, \dots, s_{(b,n-1)}\}$.
Conversely, for $j < n$ we have
\begin{equation}\label{gen5.1}
s_{(a,j)} = \begin{cases}
    U e_{n,1} &\text{ if $j = n-1$}\\
    e_{j+1,j+2}, & \text{ otherwise,}
    \end{cases}\quad\text{ and }\quad
s_{(b,j)} = \begin{cases}
    V e_{n,1} &\text{ if $j = n-1$}\\
    e^{-2\pi i\theta}U^*Ve_{j+1,j+2} &\text{ otherwise.}
    \end{cases}
\end{equation}
Hence $C^*(\Lambda_v)$ is generated by $U, V$ and the $e_{i,j}$.

The universal property of $A_{n\theta} \otimes M_n(\CC)$ gives a surjective homomorphism
$\phi\colon A_{n\theta} \otimes M_{n}(\CC) \to C^*(\Lambda,c)$ such that
\[
\phi(u\otimes 1_n) = U, \quad \phi(v\otimes 1_n) = V,\quad\text{ and }\quad
    \phi(1\otimes \zeta_{i,j})=e_{i,j}.
\]
The formulas~\eqref{gen5.1} describe a Cuntz-Krieger $\phi$-representation in the sense
of \cite[Definition 7.4]{KumPasSim3}, so the universal property of $C^*_\phi(\Lambda)$
gives an inverse for $\phi$, and so $\phi$ is an isomorphism\footnote{Of course, if
$\theta$ is irrational then $A_\theta$ and hence $A_\theta \otimes M_n(\CC)$ is simple,
and so $\phi$ is automatically injective.}. Hence $A_{n\theta}\otimes M_{n}(\CC)\cong
C^*_\phi(\Lambda)\cong C^*(\Lambda, c)$ by \cite[Corollary 5.7]{KumPasSim2}.
\end{proof}

\begin{lemma}
\label{lem5.11} Let $\Lambda=\Lambda_E$ be the rank-3 Bratteli diagram associated to a
singly connected weighted Bratteli diagram $E$ together with a compatible collection
$(c_v)$ of 2-cocycles. Take $e \in E^1$, and suppose that $c_{r(e)} = \dstar{}c_\theta$
and $c_{s(e)} = \dstar{}c_\theta$ where $\theta\in \RR\backslash\QQ$. Let $n=w(r(e))$,
$m=w(s(e))$ and $l = m/n$. Let $\iota_e \colon C^*(\Lambda_{r(e)}, c_{r(e)}) \to
C^*(\Lambda_{s(e)}, c_{s(e)})$ be as in Lemma~\ref{lem5.5}, and let $\phi_{r} \colon
A_{n\theta}\otimes M_{n}(\CC) \to C^*(\Lambda_{r(e)}, c_{r(e)})$ and $\phi_{s} \colon
A_{m\theta}\otimes M_{m}(\CC) \to C^*(\Lambda_{s(e)}, c_{s(e)}) $ be the isomorphisms
obtained from Lemma~\ref{lem5.7}. Let $u_r$ and $v_r$ be the generators of $A_{n\theta}$,
and let $\rho_r \colon K_1(A_{n\theta} \otimes M_{n}(\CC)) \to \ZZ^2$ be the isomorphism
such that $\rho_r([u_r \oplus 1_{n-1}]) = \minivector{1}{0}$ and $\rho_r([v_r \oplus
1_{n-1}]) = \minivector{0}{1}$; define $u_s, v_s$ and $\rho_s \colon K_1(A_{m\theta}
\otimes M_{m}(\CC)) \to \ZZ^2$ similarly. Then the diagram
\begin{equation}\label{eq:K1cd}
\xymatrix{K_1(C^*(\Lambda_{r(e)}, c_{r(e)})) \ar[r]^{K_1(\iota_e)}
    & K_1(C^*(\Lambda_{s(e)}, c_{s(e)}))\\
K_1(A_{n\theta}\otimes M_{n}(\CC)) \ar[d]^{\rho_r} \ar[u]_{K_1(\phi_{r})}
    & K_1(A_{m\theta}\otimes M_{m}(\CC)) \ar[d]^{\rho_s} \ar[u]_{K_1(\phi_{s})}\\
\ZZ^2 \ar[r]^{\minimatrix{1}{1-l}{0}{l}} & \ZZ^2}
\end{equation}
commutes.
\end{lemma}	
\begin{proof}
Recall $n=w(r(e))$, $m=w(s(e))$, and $l = m/n$. Since the maps $a\mapsto a\oplus 1_{n-1}$
and $b \mapsto b \oplus 1_{m-1}$ induce isomorphisms in $K_1$, the classes $[u_r\oplus
1_{n-1}]$ and $[v_r\oplus 1_{n-1}]$ generate $K_1(A_{n\theta}\otimes M_{n}(\CC))$ and
$[u_s\oplus 1_{m-1}]$ and $[v_s\oplus 1_{m-1}]$ generate $K_1(A_{m\theta}\otimes
M_{m}(\CC))$.

We claim that $K_1(\iota_e)$ maps $[\phi_r(u_r\oplus 1_{n-1})]$ to $[\phi_s(u_s\oplus
1_{m-1})]$. To see this, for $i < n$, let $\mu^r_{i}$ denote the unique element of
$(r(e),i)\Lambda_{r(e)}^{(n,0)}$, and let
\[
    U_r := \sum_{i=0}^{n-1} s_{\mu^r_{i}} \in C^*(\Lambda_{r(e)},c_{r(e)}).
\]
By Lemma~\ref{lem5.7}, $U_r = \phi_r(u_r\otimes 1_n)$, and so
\[
U_r = \phi_r(u_r\otimes 1)
    = \prod^{n-1}_{i=0} \phi_r(1_i \oplus u_r \oplus 1_{n-i-1}).
\]
Hence $[U_r] = \sum^{n-1}_{i=0} [\phi_r(1_i \oplus u_r \oplus 1_{n-i-1})] =
n[\phi_r(u_r\oplus 1_{n-1})]$. An identical argument shows that if $\mu^s_{i}$ denotes the
unique element of $(s(e),i)\Lambda_{s(e)}^{(m,0)}$ for $i < m$, then $U_s :=
\sum_{i=0}^{m-1} s_{\mu^s_{i}}$ satisfies $[U_s] = m[\phi_s(u_s\oplus 1_{m-1})]$. Direct computation using that
$ln = m$ shows that $\iota_e$ maps $(\sum_{i=0}^{n-1}s_{\mu^r_{i}})^l$ to
$\sum_{i=0}^{m-1} s_{\mu^s_{i}}$ . We have
\[
m[\phi_r(u_r \oplus 1_{n-1})]
    = ln[\phi_r(u_r \oplus 1_{n-1})]
    =l[U_r]
    =\Big[\Big(\sum_{i=0}^{n-1} s_{\mu^r_{i}}\Big)^l\Big].
\]
Hence
\[
K_1(\iota_e)(m[\phi_r(u_r \oplus 1_{n-1})]) =
    \Big[\sum_{i=0}^{m-1} s_{\mu^s_{i}}\Big]
    =[U_s]
    =m[\phi_s(u_s \oplus 1_{m-1})].
\]
Since $\theta \in \RR \backslash \QQ$, we have $K_1(A_{n\theta}\otimes M_n(\CC))=\langle
[u_r \oplus 1_{n-1}],[v_r \oplus 1_{n-1}]\rangle \cong\ZZ^2$, and we deduce that
$K_1(\iota_e)$ maps $[\phi_r(u_r \oplus 1_{n-1})]$ to $[\phi_s(u_s \oplus 1_{m-1})]$.

We now show that $K_1(\iota_e)$ maps $[\phi_r(v_r \oplus 1_{n-1})]$ to $l[\phi_s(v_s
\oplus 1_{m-1})]-(l-1)[\phi_s(u_s \oplus 1_{m-1})]$. Let $\nu^r_{i}, \nu^s_j$ be the
unique elements of $(r(e),i)\Lambda_{r(e)}^{(n-1,1)}$ and
$(s(e),j)\Lambda_{s(e)}^{(m-1,1)}$ for each $i,j$. Let
$W_r=\sum_{i=0}^{n-1}s_{\nu^r_{i}}$. Fix $\alpha_0, \dots, \alpha_{n-1}$ in $\TT$ such
that $e^{2\pi i\theta}\alpha_{i-1}=\alpha_i$, and set $V_r=\sum_{i=0}^{n-1} \alpha_i
s_{\nu^r_{i}}$. Lemma~\ref{lem5.7} gives $V_r = \phi_{r}(v_r\otimes 1_n)$. Since $V_r =
W_r\cdot \sum_{i=0}^{n-1} \alpha_i s_{(r(e),i)}$, and since $[1]=[\sum_{i=0}^{n-1}
\alpha_i s_{(r(e),i)}]$ it follows that $n[\phi_r(v_r \oplus 1_{n-1})] = [V_r] = [W_r]$.
Similarly $m[\phi_s(v_s \oplus 1_{m-1})] = [\sum_{i=0}^{m-1}s_{\nu^s_{i}}]$, and
\begin{align*}
n[\phi_r(v_r \oplus 1_{n-1})]&{} + (l-1)n[\phi_r(u_r \oplus 1_{n-1})] \\
    &= [W_r]+(l-1)[U_r]
    = \Big[\Big(\sum_{i=0}^{n-1} s_{\nu^r_{i}}\Big)
            \Big(\sum_{i=0}^{n-1} s_{\mu^r_{i}}\Big)^{l-1}\Big]
    = n\Big[\sum_{i=0}^{n-1} s_{\nu^r_i(\mu^r_i)^{l-1}}\Big].
\end{align*}
For $j < n$, we have $p_e^{-1}(\nu^r_i(\mu^r_i)^{l-1}) = \{\nu^s_i, \nu^s_{i+n}, \dots,
\nu^s_{i + (l-1)n}\}$. Hence
\begin{align*}
K_1(\iota_e)\big(n[\phi_r(v_r \oplus 1_{n-1})]&{} + (l-1)n[\phi_r(u_r \oplus 1_{n-1})]\big)\\
    &= n \big[\iota_e(s_{\nu^r_i(\mu^r_i)^{l-1}})\big]
    = \Big[\sum_{i=0}^{m-1} s_{\nu^s_{i}}\Big]
    = m[\phi_s(v_s \oplus 1_{m-1})].
\end{align*}
Since $m = nl$, we deduce that $K_1(\iota_e)$ sends $[\phi_r(v_r \oplus 1_{n-1})] +
(l-1)[\phi_r(u_r \oplus 1_{n-1})]$ to $l[\phi_s(v_s \oplus 1_{m-1})]$. We saw above that
$K_1(\iota_e)\big((l-1)[\phi_r(u_r \oplus 1_{n-1})]\big) = (l-1)[\phi_s(u_s \oplus
1_{m-1})]$, so subtracting gives $K_1(\iota_e)\big([\phi_r(v_r \oplus 1_{n-1})]\big) =
l[\phi_s(v_s \oplus 1_{m-1})] - (l-1)[\phi_s(u_s \oplus 1_{m-1})]$. So the
diagram~\eqref{eq:K1cd} commutes as claimed.
\end{proof}

Let $T(A)$ denote the set of \emph{tracial states}, i.e., positive linear functionals
with the trace property and norm one, on a $C^*$-algebra $A$. For any $\tau\in T(A)$
there is a map $K_0(\tau)\colon K_0(A)\to \RR$ such that $K_0(\tau)([p] -
[q])=\sum_i\tau(p_{ii} - q_{ii})$ for any projections $p,q\in M_n(A)$. When $\tau = \Tr$,
the unique tracial state on $A_{\theta}\otimes M_k(\CC)$ for $\theta\in
\RR\backslash\QQ$, the map $K_0(\tau)$ is an order isomorphism of $K_0(A_\theta \otimes
M_k(\CC))$ onto $\frac{1}{k}\ZZ+\frac{\theta}{k}\ZZ$ \cite{PimVoi}.

\begin{lemma} \label{lem5.12}
Let $\Lambda=\Lambda_E$ be the rank-3 Bratteli diagram associated to a singly connected
weighted Bratteli diagram $E$ together with a compatible collection $(c_v)$ of
2-cocycles. Take $e \in E^1$ and suppose that $c_{r(e)}=\dstar{}c_\theta$ and
$c_{s(e)}=\dstar{}c_\theta$ where $\theta\in \RR\backslash\QQ$. Let $n=w(r(e))$ and
$m=w(s(e))$. Let $\iota_e \colon C^*(\Lambda_{r(e)}, c_{r(e)}) \to C^*(\Lambda_{s(e)},
c_{s(e)})$ be as in Lemma~\ref{lem5.5}, and let $\phi_{r} \colon A_{n\theta}\otimes
M_{n}(\CC) \to C^*(\Lambda_{r(e)}, c_{r(e)})$ and $\phi_{s} \colon A_{m\theta}\otimes
M_{m}(\CC) \to C^*(\Lambda_{s(e)}, c_{s(e)})$ be the isomorphisms obtained from
Lemma~\ref{lem5.7}. Let $\tau_r$ and $\tau_s$ be the unique tracial states on
$A_{n\theta}\otimes M_{n}(\CC)$ and $A_{m\theta}\otimes M_{m}(\CC)$. Then the diagram
\begin{equation}\label{eq:K0cd}
\xymatrix{K_0(C^*(\Lambda_{r(e)}, c_{r(e)})) \ar[r]^{K_0(\iota_e)}
    & K_0(C^*(\Lambda_{s(e)}, c_{s(e)}))\\
K_0(A_{n\theta}\otimes M_{n}(\CC)) \ar[d]^{K_0(\tau_r)} \ar[u]_{K_0(\phi_{r})}
    & K_0(A_{m\theta}\otimes M_{m}(\CC)) \ar[d]^{K_0(\tau_s)} \ar[u]_{K_0(\phi_{s})}\\
\frac{1}{n}\ZZ + \theta \ZZ \ar[r]^{\subseteq} & \frac{1}{m} \ZZ + \theta \ZZ}
\end{equation}
commutes.
\end{lemma}	
\begin{proof}
Define $l = m/n$. Let $p_r$ and $p_s$  denote the Powers-Rieffel projections in
$A_{n\theta}$ and $A_{m\theta}$ respectively. Since the maps $a\mapsto a\oplus 0_{n-1}$
and $b \mapsto b \oplus 0_{m-1}$ induce an isomorphisms in $K_0$, the elements $[1\oplus
0_{n-1}]$ and $[p_r\oplus 0_{n-1}]$ generate $K_0(A_{n\theta}\otimes M_{n}(\CC))$, and
$[1\oplus 0_{m-1}]$ and $[p_s\oplus 1_{m-1}]$ generate $K_0(A_{m\theta}\otimes
M_{m}(\CC))$. Then $K_0(\iota_e)$ maps $[\phi_r(1\oplus 0_{n-1})]$ to $l[\phi_s(1\oplus
0_{m-1})]$ because
\[
K_0(\iota_e)\big(n[\phi_r(1\oplus 0_{n-1})]\big)
    = K_0(\iota_e)\big([1_{C^*(\Lambda_{r(e)})}]\big)
    = [1_{C^*(\Lambda_{s(e)})}]
    = m[\phi_r(1\oplus 0_{m-1})],
\]
and $m = nl$.

We show that $K_0(\iota_e)$ maps $[\phi_r(p_r \oplus 0_{n-1})]$ to $[\phi_s(p_s \oplus
0_{m-1})]$. Let $\tilde\tau_r := \tau_r \circ \phi^{-1}_r$ and $\tilde\tau_s := \tau_s
\circ \phi^{-1}_s$ be the unique tracial states on $C^*(\Lambda_{r(e)}, c_{r(e)})$ and
$C^*(\Lambda_{s(e)}, c_{s(e)})$. Since $\iota_e(C^*(\Lambda_{r(e)}, c_{r(e)})) \subseteq
C^*(\Lambda_{s(e)}, c_{s(e)})$ uniqueness of $\tilde\tau_r$ implies that $\tilde\tau_s
\circ \iota_e = \tilde\tau_r$, and since the unique tracial state $\Tr$ on $A_{n\theta}$
satisfies $\Tr(p_r) = n\theta$, we have $\tau_r(p_r\oplus 0_{n-1}) = \theta$. Hence
\[
\tilde\tau_s\circ \iota_e(\phi_r(p_r\oplus 0_{n-1}))
    =\tilde\tau_r(\phi_r(p_r\oplus 0_{n-1}))
    =\tau_r(p_r \oplus 0_{n-1})
    =\theta
    =\tilde\tau_s(\phi_s(p_s\oplus 0_{m-1})).
\]
In particular
\[
K_0(\tilde\tau_s \circ \iota_e)([\phi_r(p_r \oplus 0_{n-1}])
    = K_0(\tilde\tau_s)([\phi_s(p_s \oplus 0_{m-1})]).
\]
Since $\theta$ is irrational, $K_0(\tilde\tau_s)$ is an isomorphism, and so we deduce
that $K_0(\iota_e)\big([\phi_r(p_r \oplus 0_{n-1})]\big) = [\phi_s(p_s \oplus 0_{m-1})]$,
and that the diagram~\eqref{eq:K0cd} commutes.
\end{proof}

\begin{cor} \label{cor5.13}
Let $\Lambda=\Lambda_E$ be the rank-3 Bratteli diagram associated to a
singly connected weighted Bratteli diagram $E$ together with a compatible collection
$(c_v)$ of 2-cocycles. Let $v \in E^0$ and suppose that $c_v=\dstar{}c_\theta$ for some
$\theta\in \RR\backslash\QQ$. Then the ordered $K$-theory of $C^*(\Lambda_v,c_v)$ is
given by
\[
(K_0,K_0^+,K_1)
    =\Big(\frac{1}{w(v)}\ZZ+\theta\ZZ,
        \Big(\frac{1}{w(v)}\ZZ+\theta\ZZ\Big) \cap [0,\infty),
        \ZZ^2\Big).
\]
\end{cor}
\begin{proof}
The irrational rotation algebra $A_\theta$ is a stably finite unital exact $C^*$-algebra
with a simple, weakly unperforated $K_0$-group \cite{PimVoi, BlaKumRor} (see also
\cite[p.~36]{Bla}). Hence \cite[p.~42]{Bla} gives
\[
K_0(C^*(\Lambda_v,c_v))^+=\{0\}\cup\{x: K_0(\tau)(x)>0
    \text{ for all } \tau\in T(C^*(\Lambda_v,c_v))\}.
\]
Since $C^*(\Lambda_v,c_v) \cong A_\theta\otimes M_k(\CC)$ admits a unique tracial state
$\tau$, the map $K_0(\tau)\colon K_0(A_\theta\otimes M_k(\CC)) \to \RR^+$ is an order
isomorphism onto its range, and so $K_0(A_\theta \otimes M_k(\CC))^+ =
(\frac{1}{k}\ZZ+\frac{\theta}{k}\ZZ) \cap [0,\infty)$. The result follows.
\end{proof}

\begin{proof}[Proof of Theorem~\ref{cor5.8}]
Theorem~\ref{thm5.6} shows that
\begin{align*}
K_*(C^*(\Lambda, c))
    &\cong K_*(P_0 C^*(\Lambda) P_0) \\
    &\cong \varinjlim \bigg(\bigoplus_{v\in E^0_n} K_*(M_{E^0_1E^*v}(C^*(\Lambda_v, c_v))),\\
        &\qquad\qquad\qquad K_*\Big(\sum_{v\in E^0_n}a_v \mapsto \sum_{w\in E^0_{n+1}}
            \operatorname{diag}_{{e\in E^1w}} (\iota_e(a_{r(e)}))\Big)\bigg).
\end{align*}
Each $K_*(M_{E^0_1E^*v}(C^*(\Lambda_v, c_v))) \cong K_*(C^*(\Lambda_v, c_v))$ and these
isomorphisms are compatible with the connecting maps. Lemma~\ref{lem5.7} shows that each
$C^*(\Lambda_v, c_v) \cong A_{w(v)\theta} \otimes M_{w(v)}(\CC)$ and hence has $K$-theory
$(\frac{1}{w(v)}\ZZ + \theta \ZZ, \ZZ^2)$, and Lemmas \ref{lem5.11}~and~\ref{lem5.12}
show that the connecting maps are as claimed. The order on $K_0$ follows from
Corollary~\ref{cor5.13}.

For the final statement, observe that under the canonical isomorphisms $\frac{1}{w(v)}\ZZ
+ \theta \ZZ \cong \ZZ^2$, the inclusion maps $A_n(v,u)$ of Remark~\ref{rmk:matrices of
maps} are implemented by the matrices $\left(\begin{smallmatrix}w(u)/w(v) & 0 \\0
&1\end{smallmatrix}\right)$ (see also the proof of Lemma~\ref{lem5.12}). The
corresponding maps $B_n(v,u)$ are implemented by the matrices $\left(\begin{smallmatrix}1 & 1 - w(u)/w(v) \\
0 & 1\end{smallmatrix}\right)$. Fix $e \in E^1$, let $v = r(e)$ and $u = s(e)$ and $l =
w(u)/w(v)$, and calculate:
\[
\left(\begin{matrix}
    0 & 1 \\
    1 & 1
\end{matrix}\right)
\left(\begin{matrix}
    1 & 1 - l \\
    0 & l
\end{matrix}\right)
    =
    \left(\begin{matrix}
        0 & l \\
        1 & 1
    \end{matrix}\right)
    =
    \left(\begin{matrix}
        l & 0 \\
        0 & 1
    \end{matrix}\right)
    \left(\begin{matrix}
        0 & 1 \\
        1 & 1
    \end{matrix}\right)
\]
So the automorphisms $T_n := \bigoplus_{v \in E^n_0} \left(\begin{smallmatrix}0 & 1 \\
1 & 1\end{smallmatrix}\right)\colon\bigoplus_{v \in E_n^0} \ZZ^2 \to \bigoplus_{v \in
E_n^0} \ZZ^2$ satisfy $T_n B_n = A_n T_n$ and so there is a group isomorphism
$\varinjlim(\bigoplus_{v \in E_n^0} \ZZ^2, B_n) \cong \varinjlim(\bigoplus_{v \in E_n^0}
\ZZ^2, A_n)$ that carries each $(a,b)\delta_v$ to $(b, a+b)\delta_v$ according to the
communing diagram
\begin{equation*}
\xymatrix{\ZZ^2\delta_v \ar[r]^{\cong} & \frac{1}{w(v)} \ZZ + \theta\ZZ \ar[rrr]^{A_n(v,u)}
    & & & \frac{1}{w(v)} \ZZ + \theta\ZZ \ar[r]^{\cong} & \ZZ^2\delta_u \\
\ZZ^2\delta_v \ar[rrrrr]^{B_n(v,u)} \ar[u]^{T_n(v,v)}
    & & & & &\ZZ^2\delta_u  \ar[u]^{T_n(u,u)}}.
\end{equation*}
After identifying each $\ZZ^2 \delta_v$ with $\frac{1}{w(v)} \ZZ + \theta\ZZ$ as above,
we obtain the desired isomorphism $K_1(C^*(\Lambda), c) \cong K_0(C^*(\Lambda), c)$.
\end{proof}

Having computed the $K$-theory of the $C^*(\Lambda, c)$, we conclude by observing that
they are all classifiable by their $K$-theory. We say that a weighted Bratteli diagram is
cofinal if the underlying Bratteli diagram is cofinal.

\begin{cor}\label{cor:classifiable}
Let $\Lambda=\Lambda_E$ be the rank-3 Bratteli diagram associated to a singly connected
weighted Bratteli diagram $E$. Take $\theta \in \RR \backslash \QQ$, let $c^3_\theta \in
Z^2(\ZZ^3, \TT)$ be as in~\eqref{ck}, and let $c=\dstar{}c^3_\theta\in Z^2(\Lambda,\TT)$.
Then $C^*(\Lambda, c)$ is an $A\TT$-algebra of real rank zero, and is simple if and only
if $E$ is cofinal, in which case it is classified up to isomorphism by ordered $K$-theory
and scale.
\end{cor}
\begin{proof}
Each $A_{w(v)\theta}$ is an A$\TT$ algebra \cite{EllEva}, and has real rank zero since it
has a unique trace (see, for example, \cite[Theorem~1.3]{BlaBraEllKum}). Since direct
limits of A$\TT$ algebras are A$\TT$ and since direct limits of $C^*$-algebras of real
rank zero also have real rank zero, $C^*(\Lambda, c)$ is also an A$\TT$-algebra of real
rank zero.

It is straightforward to verify that $E$ is cofinal if and only if $\Lambda$ is cofinal.
Hence \cite[Lemma~7.2]{SimWhiWhi} implies that if $E$ is not cofinal then $C^*(\Lambda,
c)$ is not simple. Now suppose that $E$ is cofinal. Following the argument of
\cite[Proposition~5.1]{BatPasRaeSzy} shows that any ideal of $C^*(\Lambda,c)$ which
contains some $s_v$ is all of $C^*(\Lambda,c)$. So if $\psi$ is a nonzero homomorphism of
$C^*(\Lambda, c)$, then $\psi(s_v) \not= 0$ for all $v$. That is $\psi|_{C^*(\Lambda_v,
c)}$ is nonzero for each $v \in E^0$. But each $C^*(\Lambda_v, c) \cong A_{w(v)\theta}
\otimes M_{w(v)}(\CC)$ is simple, and it follows that $\psi$ is injective on each
$C^*(\Lambda_v, c)$ and hence isometric on each $\bigoplus_{v \in E^0_n} C^*(\Lambda_v,
c)$. So $\psi$ is isometric on a dense subspace of $C^*(\Lambda, c)$ and hence on all of
$C^*(\Lambda, c)$. Thus $C^*(\Lambda, c)$ is simple.

The final assertion follows from Elliott's classification theorem \cite{Ell}.
\end{proof}

\begin{cor}\label{cor:range}
Let $\Lambda=\Lambda_E$ be the rank-3 Bratteli diagram associated to a singly connected
cofinal weighted Bratteli diagram $E$. Take $\theta \in \RR \backslash \QQ$, let
$c^3_\theta \in Z^2(\ZZ^3, \TT)$ be as in~\eqref{ck}, and let $c=\dstar{}c^3_\theta\in
Z^2(\Lambda,\TT)$. Then there is a rank-2 Bratteli diagram $\Gamma$ (as described in
\cite[Definition~4.1]{PasRaeRorSim}) such that $C^*(\Gamma)$ is Morita equivalent to
$C^*(\Lambda, c)$.
\end{cor}
\begin{proof}
We have seen above that $C^*(\Lambda, c)$ is a simple A$\TT$ algebra of real rank zero.
We claim that $K_0(C^*(\Lambda, c))$ is a Riesz group in the sense of
\cite[Section~1]{EffrosHandelmanEtAl:AJM1980}. To see this, observe that it is clearly a
countably group satisfying $n a \ge 0$ implies $a \ge 0$ for all $a \in K_0(C^*(\Lambda,
c))$, Fix finite sets $\{a_i : i \in I\}$ and $\{b_j : j \in J\}$ of elements of $G$ such
that $a_i \le b_j$ for all $i,j$; we must find $c$ such that $a_i \le c \le b_j$ for all
$i,j$. We may assume that the $a_i, b_j$ all belong to some fixed $\oplus_{v \in E_n^0}
\frac{1}{w(v)}\ZZ + \theta \ZZ$, and since the order on this group is the coordinatewise
partial order, it suffices to suppose that they all belong to some fixed
$(\frac{1}{w(v)}\ZZ + \theta \ZZ)\delta_v$; but this is a totally ordered subgroup of
$\RR$, so we can take $c = \max_i a_i$.

It now follows from \cite[Theorem~2.2]{EffrosHandelmanEtAl:AJM1980} that
$K_0(C^*(\Lambda, c))$ is a dimension group. We claim that it is simple. Indeed, suppose
that $J$ is a nontrivial ideal of $K_0(C^*(\Lambda, c))$. Then each $J_v := J \cap
(\frac{1}{w(v)}\ZZ + \theta \ZZ)\delta_v$ is an ideal in this subgroup, and therefore the
whole subgroup since each $\frac{1}{w(v)}\ZZ + \theta \ZZ$ is a simple dimension group.
Since $J$ is nontrivial, we may fix $v \in E^0$, say $v \in E^0_p$ such that $J_v \not=
\emptyset$, and therefore $J_v = (\frac{1}{w(v)}\ZZ + \theta \ZZ)\delta_v$. Choose $v'
\in E^0$, say $v' \in E_m^0$; we just have to show that $J_{v'}$ is nontrivial. Since $E$
is cofinal, there exists $n$ sufficiently large so that $s(v' E^n) \subseteq s(vE^*)$
(see, for example, \cite[Proposition~A.2]{LewinSims:MPCPS2010}). It follows that the
element $1 \delta_{v'}$ of $(\frac{1}{w(v')}\ZZ + \theta \ZZ)\delta_{v'}$ satisfies $1
\delta_{v'} = \sum_{\mu \in v'E^n} 1 \delta_{s(\mu)}$. Let $N := |v' E^n|$. Then
\[
N \delta_v
    = N \Big(\sum_{\nu \in vE^{m+n-p}} 1 \delta_{s(\nu)}\Big)
    \ge \sum_{u \in E^0_{m+n}, vE^*u \not= \emptyset} N\delta_u
    \ge \delta_{v'}.
\]
Since $N \delta_v \in J$ and $J$ is an ideal of the Riesz group $K_0(C^*(\Lambda, c))$,
it follows that $\delta_{v'} \in J$. Hence $K_0(C^*(\Lambda, c))$ is a simple dimension
group.

For any $v$, we have $(\frac{1}{w(v)}\ZZ + \theta \ZZ)\delta_v \cong \ZZ^2$ as a
group, so $K_0(C^*(\Lambda, c))$ is not $\ZZ$. Now the argument of the proof of
\cite[Theorem~6.2(2)]{PasRaeRorSim} shows that there is a sequence of proper nonnegative
matrices $A'_n \in M_{q_n, q_{n+1}}(\NN)$ such that $K_0(C^*(\Lambda, c)) = \varinjlim
(\ZZ^{q_n}, A'_n)$. We may now apply \cite[Theorem~6.2(2)]{PasRaeRorSim} with $B_n = A_n
= A'_n$ and $T_n = \id_{q_n}$ for all $n$ to see that there is a rank-2 Bratteli diagram
$\Gamma$ such that $C^*(\Gamma)$ is simple and has real rank zero and ordered $K$-theory
is identical to that of $C^*(\Lambda, c)$. So the two are Morita equivalent by
Corollary~\ref{cor:classifiable}.
\end{proof}

\section{Examples}\label{sec6}

In this section we present a few illustrative examples of our $K$-theory calculations
from the preceding section.

\begin{eg}
Consider the rank-3 Bratteli diagram $\Lambda$ associated to the singly connected
weighted Bratteli diagram $E$ pictured below.
$$
\parbox{5cm}{
\begin{tikzpicture}[scale=1.5]
    \node[inner sep=0.5pt, circle] (31) at (0,2) {\small$\bullet$};
    \node[anchor=south, inner sep=1pt] at (31.north) {\tiny$1$};
    \node[inner sep=0.5pt, circle] (32) at (1,2) {\small$\bullet$};
    \node[anchor=south, inner sep=1pt] at (32.north) {\tiny$2$};
    \node[inner sep=0.5pt, circle] (33) at (2,2) {\small$\bullet$};
    \node[anchor=south, inner sep=1pt] at (33.north) {\tiny$4$};
    \node[inner sep=0.5pt, circle] (34) at (3,2) {$\cdots$};
    \draw[-latex, black] (32)--(31);
    \draw[-latex, black] (34)--(33);
    \draw[-latex, black] (33)--(32);
\end{tikzpicture}}
$$
Let $v_n$ be the vertex at level $n$, and let $e_n$ denote the unique edge with range
$v_n$. Let $\theta \in \RR \setminus \QQ$, let $c^3_\theta \in Z^2(\ZZ^3, \TT)$ be as
in~\eqref{ck}, and let $c=\dstar{}c^3_\theta\in Z^2(\Lambda,\TT)$; this is (up to
cohomology) the unique 2-cocycle extending $c_1=\dstar{}c_\theta\in Z^2(\Lambda_1,\TT)$.
By Theorem~\ref{thm5.6} the twisted 3-graph $C^*$-algebra $C^*(\Lambda, c)$ is Morita
equivalent to $\varinjlim \left(C^*(\Lambda_{v_n}, c_{v_n}),\ \iota_{e_n}\right)$. Hence
\[
K_0(C^*(\Lambda, c)) \cong \bigcup_n \big(\frac{1}{2^n}\ZZ + \theta\ZZ\big),
\]
with positive cone $(\ZZ[\frac{1}{2}]+\theta\ZZ) \cap [0,\infty)$, and $K_1(C^*(\Lambda,
c)) \cong K_0(C^(\Lambda, c))$ as groups.
\end{eg}

\begin{eg}
Let $\Lambda = \Lambda_E$ be the rank-3 Bratteli diagram associated to the singly
connected weighted Bratteli diagram $E$ given by
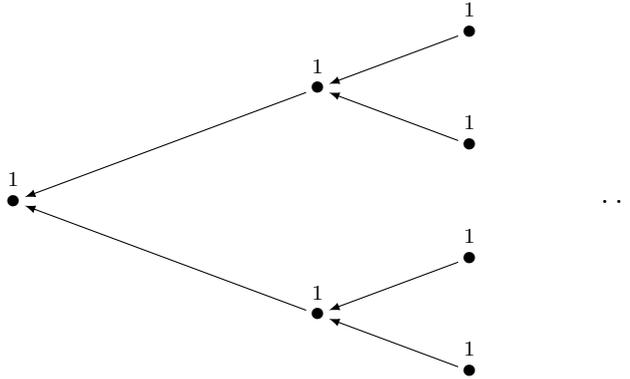
\[
\begin{tikzpicture}[xscale=2, yscale=1.5]
    \node[inner sep=0.5pt, circle] (0) at (0,0) {\small$\bullet$};
    \node[anchor=south, inner sep=1pt] at (0.north) {\tiny$1$};
    \node[inner sep=0.5pt, circle] (00) at (2,1) {\small$\bullet$};
    \node[anchor=south, inner sep=1pt] at (00.north) {\tiny$1$};
    \node[inner sep=0.5pt, circle] (01) at (2,-1) {\small$\bullet$};
    \node[anchor=south, inner sep=1pt] at (01.north) {\tiny$1$};
    \node[inner sep=0.5pt, circle] (000) at (3,1.5) {\small$\bullet$};
    \node[anchor=south, inner sep=1pt] at (000.north) {\tiny$1$};
    \node[inner sep=0.5pt, circle] (001) at (3,0.5) {\small$\bullet$};
    \node[anchor=south, inner sep=1pt] at (001.north) {\tiny$1$};
    \node[inner sep=0.5pt, circle] (010) at (3,-0.5) {\small$\bullet$};
    \node[anchor=south, inner sep=1pt] at (010.north) {\tiny$1$};
    \node[inner sep=0.5pt, circle] (011) at (3,-1.5) {\small$\bullet$};
    \node[anchor=south, inner sep=1pt] at (011.north) {\tiny$1$};
    \node at (4,0) {\dots};
    \draw[-latex, black] (00)--(0);
    \draw[-latex, black] (01)--(0);
    \draw[-latex, black] (000)--(00);
    \draw[-latex, black] (001)--(00);
    \draw[-latex, black] (010)--(01);
    \draw[-latex, black] (011)--(01);
\end{tikzpicture}
\]
For each $n\geq 1$ let $v_{n,j}$ denote the $j$'th vertex of $E$ at level $n$ counting
from top to bottom, and for each $n\geq 2$ let $e_{n,j}$ denote the unique edge of source
$v_{n,j}$. Let $\theta \in \RR \setminus \QQ$, let $c^3_\theta \in Z^2(\ZZ^3, \TT)$ be as
in~\eqref{ck}, and let $c=\dstar{}c^3_\theta\in Z^2(\Lambda,\TT)$. Then
\[
K_1(C^*(\Lambda, c)) \cong K_0(C^*(\Lambda, c))
    \cong \left(\bigoplus^{2^{n-1}}_{i=1} (\ZZ + \theta \ZZ), a \mapsto (a,a)\right),
\]
which is isomorphic to $(\ZZ+\theta\ZZ)^{\infty} \subseteq \RR^\infty$, with positive
cone carried to $(\ZZ+\theta\ZZ)^{\infty} \cap [0,\infty)^\infty$.
\end{eg}

\begin{eg}
Consider the rank-3 Bratteli diagram $\Lambda$ associated to the singly connected
weighted Bratteli diagram $E$ pictured below.
$$
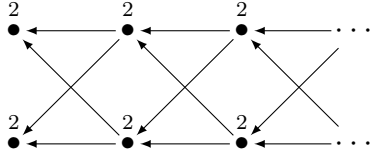
\parbox{5cm}{
\begin{tikzpicture}[scale=1.5]
    \node[inner sep=0.5pt, circle] (21) at (0,1) {\small$\bullet$};
    \node[anchor=south, inner sep=1pt] at (21.north) {\tiny$2$};
    \node[inner sep=0.5pt, circle] (22) at (1,1) {\small$\bullet$};
    \node[anchor=south, inner sep=1pt] at (22.north) {\tiny$2$};
    \node[inner sep=0.5pt, circle] (23) at (2,1) {\small$\bullet$};
    \node[anchor=south, inner sep=1pt] at (23.north) {\tiny$2$};
    \node[inner sep=0.5pt, circle] (24) at (3,1) {$\cdots$};
    \node[inner sep=0.5pt, circle] (31) at (0,2) {\small$\bullet$};
    \node[anchor=south, inner sep=1pt] at (31.north) {\tiny$2$};
    \node[inner sep=0.5pt, circle] (32) at (1,2) {\small$\bullet$};
    \node[anchor=south, inner sep=1pt] at (32.north) {\tiny$2$};
    \node[inner sep=0.5pt, circle] (33) at (2,2) {\small$\bullet$};
    \node[anchor=south, inner sep=1pt] at (33.north) {\tiny$2$};
    \node[inner sep=0.5pt, circle] (34) at (3,2) {$\cdots$};
    \draw[-latex, black] (23)--(22);
    \draw[-latex, black] (32)--(31);
    \draw[-latex, black] (24)--(23);
    \draw[-latex, black] (34)--(33);
    \draw[-latex, black] (33)--(22);
    \draw[-latex, black] (33)--(32);
    \draw[-latex, black] (22)--(21);
    \draw[-latex, black] (23)--(32);
    \draw[-latex, black] (22)--(31);
    \draw[-latex, black] (32)--(21);
    \draw[-latex, black] (34)--(23);
    \draw[-latex, black] (24)--(33);
\end{tikzpicture}}
$$
Let $\theta \in \RR \setminus \QQ$, let $c^3_\theta \in Z^2(\ZZ^3, \TT)$ be as
in~\eqref{ck}, and let $c=\dstar{}c^3_\theta\in Z^2(\Lambda,\TT)$. Then
\[
K_1(C^*(\Lambda, c)) \cong K_0(C^*(\Lambda, c))
     \cong \varinjlim \left((\ZZ + \theta\ZZ) \oplus (\ZZ + \theta \ZZ),
        \left(\begin{smallmatrix} \id & \id \\
                                  \id & \id
        \end{smallmatrix}\right)\right),
\]
which is isomorphic to $\ZZ[\frac{1}{2}]+\ZZ[\frac{\theta}{2}]$ with positive cone
$(\ZZ[\frac{1}{2}]+\ZZ[\frac{\theta}{2}]) \cap [0,\infty)$.
\end{eg}

\section{Rank-3 Bratteli diagrams and traces}\label{sec:BDtraces}

In this section we show how to identify traces on twisted $C^*$-algebras associated to
rank-3 Bratteli diagrams.

First we briefly introduce densely defined traces on $C^*$-algebras, following
\cite{PasRenSim}. (Note that there are other definitions of a trace; see for example
\cite{KirRor}.) We let $A^+$ denote the positive cone in a $C^*$-algebra $A$, and we
extend arithmetic on $[0,\infty]$ so that $0\times \infty=0$. A \emph{trace} on a
$C^*$-algebra $A$ is an additive map $\tau:A^+\to [0, \infty]$ which respects scalar
multiplication by non-negative reals and satisfies the \emph{trace property} $\tau(a^*a)
= \tau(aa^*)$, $a\in  A$. A trace $\tau$ is \emph{faithful} if $\tau(a) = 0$ implies $a =
0$. It is \emph{semifinite} if it is finite on a norm dense subset of $A^+$ i.e,
$\overline{\{a\in A^+: 0\leq \tau(a) < \infty\}}=A^+$. A trace $\tau$ is \emph{lower
semicontinuous} if $\tau(a)\leq \lim \inf_n \tau(a_n)$ whenever $a_n\to a$ in $A^+$. We
may extend a semifinite trace $\tau$ by linearity to a linear functional on a dense
subset of $A$. The domain of definition of a densely defined trace is a two-sided ideal
$I_\tau\subset A$.

Following \cite{PasRenSim, Tom} a \emph{graph trace} on a $k$-graph $\Lambda$ is a
function $g\colon\Lambda^0\to \RR^+$ satisfying the \emph{graph trace property}
\[
g(v)=\sum_{\lambda \in v\Lambda^{\leq n}} g(s(\lambda))\quad\text{ for all $v\in \Lambda^0$ and
$n\in \NN^k$.}
\]
A graph trace is \emph{faithful} if it is non-zero on every vertex in $\Lambda$.

\begin{lemma}[\cite{PasRenSim}]\label{lem4.5}
Let $\Lambda$ be a row-finite locally convex $k$-graph, and let $c \in Z^2(\Lambda,
\TT)$. For each semifinite trace $\tau$ on $C^*(\Lambda, c)$ there is a graph trace $g$ on
$\Lambda$ such that $g(v)=\tau(s_v)$ for all $v\in \Lambda^0$.	
\end{lemma}
\begin{proof}
Fix $v\in \Lambda^0$. Since $\tau$ is semifinite we may extend it to the two-sided ideal
$I_\tau=\{a:\tau(a)\leq \infty\}$. Choose $a\in (I_\tau)_+$ such that $\|s_v-a\|< 1$.
Then $\|s_v-s_vas_v\|< 1$, and so $s_v=bs_vas_v\in I_\tau$, where $b$ is the inverse of
$s_vas_v$ in $s_vC^*(\Lambda, c)s_v$. In particular $g(v)=\tau(s_v)<\infty$. The graph
trace property follows from applying $\tau$ to~(CK4).
\end{proof}

It turns out, conversely, that each graph trace corresponds to a trace. This however
requires a bit more machinery which we now introduce. Recall that each twisted $k$-graph
$C^*$-algebra $C^*(\Lambda, c)$ carries a \emph{gauge action} $\gamma$ of $\TT^k$ such
that $\gamma_z(s_\lambda) = z^{d(\lambda)} s_\lambda$. Averaging against Haar measure
over this action gives a faithful conditional expectation $\Phi^\gamma\colon a \mapsto
\int_\TT \gamma_z(a)\,dz$ onto the fixed-point algebra $C^*(\Lambda, c)^\gamma$, which is
called the \emph{core}. We have $\Phi^\gamma(s_\mu s^*_\nu) = \delta_{d(\mu), d(\nu)}
s_\mu s^*_\nu$, and so $C^*(\Lambda, c)^\gamma = \clsp\{s_\mu s^*_\nu : d(\mu) =
d(\nu)\}$. For every finite set $F\subseteq \Lambda$ there is a smallest finite set
$F'\subseteq \Lambda$ such that $F\subseteq F'$ and $A_{F'}:=\lsp\{s_\mu s_\nu^* \in
C^*(\Lambda, c)^\gamma: \mu,\nu\in F'\}$ is a finite dimensional $C^*$-algebra
\cite[Lemma~3.2]{RaeSimYee2}. For two finite sets $F\subseteq G\subseteq \Lambda$, we
have $F' \subseteq G'$ so $A_{F'} \subseteq A_{G'}$. So any increasing sequence of finite
subsets $F_n$ such that $\bigcup_n F_n = \Lambda$ gives an AF decomposition of
$C^*(\Lambda, c)^\gamma$.

The following lemma was proved for $c = 1$ by Pask, Rennie and Sims using the augmented
boundary path representation on $\ell^2(\partial\Lambda)\otimes \ell^2(\ZZ^k)$ (see the
first arXiv version of \cite{PasRenSim}). When $c\neq 1$ it is not clear how to represent
$C^*(\Lambda, c)$ on $\ell^2(\partial\Lambda)$, so we proceed in a different way:

\begin{lemma}\label{lem3.3}
Let $\Lambda$ be a row-finite locally convex $k$-graph and let $c\in Z^2(\Lambda,\TT)$.
There is a faithful conditional expectation $E$ of $C^*(\Lambda, c)$ onto $\clsp\{s_\mu
s_\mu^*\}$ which satisfies
$$
E(s_\mu s_\nu^*) = \left\{
       \begin{array}{ll}
       s_\mu s_\mu^* & \text{ if $\mu=\nu$} \\
       0 & \text{ otherwise.}
       \end{array}
       \right.
$$
\end{lemma}
\begin{proof}
Recall that the linear map $E: M_n(\CC)\to M_n(\CC)$ such that $E(\theta_{i,j}) =
\delta_{i,j} \theta_{i,i}$ is a faithful conditional expectation.

Fix a finite set $F\subseteq \Lambda$. Select the smallest finite set $F'\subseteq
\Lambda$ such that $F\subseteq F'$ and $A_{F'}=\lsp\{s_\mu s_\nu^* : \mu,\nu\in F',
d(\mu) = d(\nu)\}$ is a finite-dimensional $C^*$-algebra. There exist integers $k_i$ and
an isomorphism $A_{F'}\cong \bigoplus_{i=1}^n M_{k_i}(\CC)$ which carries $\lsp\{s_\mu
s_\mu^*: \mu\in F'\}$ to $\lsp\{\theta_{ii}\}$, and carries each $s_\mu s_\nu^*$ with
$\mu\neq \nu$ into $\lsp\{\theta_{ij}: i\neq j\}$ \cite[Equation~(3.2)]{SimWhiWhi}. Hence
the map $s_\mu s_\nu^*\mapsto \delta_{\mu, \nu} s_\mu s_\mu^*$, from $A_{F'}$ into its
canonical diagonal subalgebra $\lsp\{s_\mu s_\mu^*: \mu\in F'\}$ is a faithful
conditional expectation. Extending this map by continuity to $C^*(\Lambda,
c)^\gamma=\overline{\bigcup_{F'}A_{F'}}$ gives a norm-decreasing linear map $\Psi :
C^*(\Lambda, c)^\gamma \to \clsp\{s_\mu s_\mu^*\}$ satisfying $\Psi(s_\mu s_\nu^*) =
\delta_{\mu,\nu} s_\mu s^*_\mu$. This $\Psi$ is an idempotent of norm one, and is
therefore a conditional expectation by \cite[Theorem~II.6.10.2]{Blackadar}. Since $\Psi$
agrees with the usual expectation of the AF-algebra $C^*(\Lambda, c)^\gamma$ onto its
canonical diagonal subalgebra, it is faithful. Hence the composition $E:=\Psi\circ
\Phi^\gamma$ is the desired faithful conditional expectation from $C^*(\Lambda, c)$ onto
$\clsp\{s_\mu s_\mu^*\}$.
\end{proof}

\begin{lemma}[{c.f. \cite[Proposition~3.10]{PasRenSim}}]\label{lem4.6}
Let $\Lambda$ be a row-finite locally convex $k$-graph and let $c\in Z^2(\Lambda,\TT)$.
For each faithful graph trace $g$ on $\Lambda$ there is a faithful, semifinite, lower
semicontinuous, gauge invariant trace $\tau_g$ on $C^*(\Lambda, c)$ such that
$\tau_g(s_\mu s_\nu^*)=\delta_{\mu,\nu} g(s(\mu))$ for all $\mu,\nu \in \Lambda$.
\end{lemma}
\begin{proof}
Take a finite $F \subseteq \Lambda$ and scalars $\{a_\mu : \mu \in F\}$. Suppose that
$\sum_{\mu \in F} a_\mu s_\mu s^*_\mu = 0$. Let $N := \bigvee_{\mu \in F} d(\mu)$.
Relation~(CK) implies that $\sum_{\mu \in F} \sum_{\alpha \in s(\mu)\Lambda^{\le N -
d(\mu)}} a_\mu s_{\mu\alpha} s^*_{\mu\alpha} = 0$. Let $G := \{\mu\alpha : \mu \in F,
\alpha \in s(\mu)\Lambda^{\le N - d(\mu)}\}$, and for $\lambda \in G$, let $b_\lambda :=
\sum_{\mu \in F, \lambda = \mu\mu'} a_\mu$. Then
\[
0 = \sum_{\mu \in F} \sum_{\alpha \in s(\mu)\Lambda^{\le N - d(\mu)}} a_\mu s_{\mu\alpha} s^*_{\mu\alpha}
    = \sum_{\lambda \in G} b_\lambda s_\lambda s^*_\lambda.
\]
Since~(CK) implies that the $s_\lambda s^*_\lambda$ where $\lambda \in G$ are mutually
orthogonal, we deduce that each $b_\lambda = 0$. Now the graph-trace property gives
\[
\sum_{\mu \in F} a_\mu g(s(\mu))
    = \sum_{\mu \in F} \sum_{\alpha \in s(\mu)\Lambda^{\le N - d(\mu)}} a_\mu g(s(\alpha))
    = \sum_{\lambda \in G} b_\lambda g(s(\lambda))
    = 0.
\]
So there is a well-defined linear map $\tau^0_g\colon\lsp\{s_\mu s^*_\mu : \mu \in
\Lambda\} \to \RR^+$ such that $\tau^0_g(s_\mu s^*_\mu) = g(s(\mu))$ for all $\mu$. Let
$E : C^*(\Lambda, c) \to \clsp\{s_\mu s^*_\mu : \mu \in \Lambda\}$ be the map of
Lemma~\ref{lem3.3}. Then $E$ restricts to a map from $A_c := \lsp\{s_\mu s^*_\nu :
\mu,\nu \in \Lambda\}$ to $\lsp\{s_\mu s^*_\mu : \mu \in \Lambda\}$. Define $\tau_g :=
\tau^0_g \circ E\colon A_c \to \lsp\{s_\mu s^*_\mu : \mu \in \Lambda\}$.

We claim that $\tau_g$ satisfies the trace condition. For this, it suffices to show that
\begin{equation}\label{eq:tracecheck}
\tau_g(s_\lambda s^*_\mu s_\eta s^*_\zeta) = \tau_g(s_\eta s^*_\zeta s_\lambda s^*_\mu)
    \quad\text{ for all $\lambda,\mu, \eta, \zeta$.}
\end{equation}
Since $E(s_\mu s^*_\nu) = 0$ unless $d(\mu) = d(\nu)$, both sides
of~\eqref{eq:tracecheck} are zero unless $d(\lambda) - d(\mu) = d(\zeta) - d(\eta)$. We
have
\begin{align}
\tau_g(s_\lambda s^*_\mu s_\eta s^*_\zeta)
    &= \sum_{(\alpha,\beta) \in \Lambda^{\min}(\mu,\eta)}
        c(\lambda,\alpha)\overline{c(\mu,\alpha)}c(\eta,\beta)\overline{c(\zeta,\beta)}
            \tau_g(s_{\lambda\alpha} s^*_{\zeta\beta}) \nonumber\\
    &= \sum_{\substack{(\alpha,\beta) \in \Lambda^{\min}(\mu,\eta) \\ \lambda\alpha=\zeta\beta}}
        c(\lambda,\alpha)\overline{c(\mu,\alpha)}c(\eta,\beta)\overline{c(\zeta,\beta)}
            g(s(\alpha)).\label{eq:LHS}
\end{align}
Similarly,
\begin{equation}\label{eq:RHS}
\tau_g(s_\eta s^*_\zeta s_\lambda s^*_\mu)
    = \sum_{\substack{(\beta,\alpha) \in \Lambda^{\min}(\zeta,\lambda) \\ \eta\beta=\mu\alpha}}
        c(\lambda,\alpha)\overline{c(\mu,\alpha)}c(\eta,\beta)\overline{c(\zeta,\beta)}
            g(s(\beta))
\end{equation}
The argument of the paragraph following Equation~(3.6) of \cite{anHLacRaeSim} shows that
$(\alpha,\beta) \mapsto (\beta,\alpha)$ is a bijection from the indexing set on the
right-hand side of~\eqref{eq:LHS} to that on the right-hand side of~\eqref{eq:RHS},
giving~\eqref{eq:tracecheck}.

We now follow the proof of Proposition~3.10 of \cite{PasRenSimArXiv}, beginning from the
second sentence, except that in the final line of the proof, we apply the gauge-invariant
uniqueness theorem \cite[Theorem~3.15]{SimWhiWhi} with $\mathcal{E} = \FE(\Lambda)$
rather than \cite[Theorem~4.1]{RaeSimYee}.
\end{proof}

\begin{thm}
\label{thm4.7} Let $\Lambda$ be a row-finite locally convex $k$-graph and let $c\in
Z^2(\Lambda,\TT)$. The map $g \mapsto \tau_g$ of Lemma~\ref{lem4.6} is a bijection
between faithful graph traces on $\Lambda$ and faithful, semifinite, lower
semicontinuous, gauge invariant traces on $C^*(\Lambda, c)$.
\end{thm}
\begin{proof}
Combine Lemma~\ref{lem4.5} and Lemma~\ref{lem4.6}.
\end{proof}

\begin{rmk}\label{rmk:nonfaithful}
If $g$ is a (not necessarily faithful) graph trace, then the graph-trace condition
ensures that $H_g := \{v \in \Lambda^0 : g(v) = 0\}$ is saturated and hereditary in the
sense of \cite[Setion~5]{RaeSimYee}, and so $\Lambda \setminus \Lambda H_g$ is also a
locally convex row-finite $k$-graph \cite[Theorem~5.2(b)]{RaeSimYee}. If $I_{H_g}$ is the
ideal of $C^*(\Lambda, c)$ generated by $\{s_v : v \in H_g\}$, then
\cite[Corollary~4.5]{SimWhiWhi} shows that $C^*(\Lambda, c)/I_{H_g}$ is canonically
isomorphic to $C^*(\Lambda \setminus \Lambda H_g, c|_{\Lambda \setminus \Lambda H_g})$.
It is easy to see that $g$ restricts to a faithful graph trace on $\Lambda \setminus
\Lambda H_g$, so Lemma~\ref{lem4.6} gives a faithful semifinite lower-semicontinuous
gauge-invariant trace on $C^*(\Lambda \setminus \Lambda H_g, c|_{\Lambda \setminus
\Lambda H_g})$. Composing this with the canonical homomorphism $\pi_{H_g}\colon
C^*(\Lambda, c) \to C^*(\Lambda \setminus \Lambda H_g, c|_{\Lambda \setminus \Lambda
H_g})$ gives a semifinite lower-semicontinuous gauge-invariant trace on $C^*(\Lambda,
c)$. So Theorem~\ref{thm4.7} remains valid if the word ``faithful" is removed throughout.
\end{rmk}

\begin{defn}
Let $(E, \Lambda, p)$ be a Bratteli diagram of covering maps between $k$-graphs. For each
$v \in E^0$, let $g_v : \Lambda_v^0 \to \RR^+$ be a graph trace. We say that the
collection $(g_v)$ of graph traces is \emph{compatible} if
\[
g_v(u) = \sum_{e \in vE^1}\sum_{p_e(w) = u} g_{s(e)}(w) \quad\text{ for all $v \in E^0$ and $u \in \Lambda^0_v$.}
\]
\end{defn}

We will show in Lemma~\ref{lem4.4} that the compatibility requirement is necessary and
sufficient to combine the $g_v$ into a graph trace on the $(k+1)$-graph $\Lambda_E$
associated to the Bratteli diagram $E$ of covering maps.

\begin{lemma}\label{lem4.3}	
A function $g\colon\Lambda^0\to \RR^+$ on the vertices of a locally convex $k$-graph
$\Lambda$ is a graph trace if and only if for all $v\in \Lambda^0$ and $i\in \{1,\dots,
k\}$ with $v\Lambda^{e_i}\neq \emptyset$ we have
\begin{equation}\label{eq4.1}
    g(v) = \sum_{\lambda \in v\Lambda^{e_i}} g(s(\lambda)).
\end{equation}
\end{lemma}
\begin{proof}
It is clear that every graph trace satisfies~\eqref{eq4.1}. The reverse implication is a
straightforward induction along the lines of, for example, the proof of \cite[Proposition
3.11]{RaeSimYee}.
\end{proof}

\begin{lemma}\label{lem4.4}
Let $\Lambda = \Lambda_E$ be the $(k+1)$-graph associated to a Bratteli diagram of
covering maps between row finite locally convex $k$-graphs. Let $g_v\colon\Lambda_v^0\to
\RR^+$ be a graph trace for each $v \in E^0$. Define $g\colon\Lambda^0\to \RR^+$ by $g
\circ \iota_v = g_v$ for all $v \in E^0$. Then $g$ is a graph trace if and only if
$(g_v)$ is compatible.
\end{lemma}
\begin{proof}
If $g$ is a graph trace, then~\eqref{eq4.1} with $i = k+1$ shows that the $g_v$ are
compatible. Conversely, if the $g_v$ are compatible, then~\eqref{eq4.1} holds for $i \le
k$ because each $g_v$ is a graph trace, and for $i=k+1$ by compatibility. So the result
follows from Lemma~\ref{lem4.3}.
\end{proof}

\begin{lemma}\label{lem:g<->h}
Let $\Lambda=\Lambda_E$ be the rank-3 Bratteli diagram associated to a singly connected
weighted Bratteli diagram $E$. Consider the Bratteli diagram $F$ such that $F^0 = E^0$
and $F^1 = \bigsqcup_{e \in E^1} \{e\} \times \ZZ/(w(s(e))/w(r(e)))\ZZ$, with $r(e,i) =
r(e)$ and $s(e,i) = s(e)$. For each graph trace $h$ on $F$, there is a graph trace $g_h$
on $\Lambda$ such that $g_h((v,j)) = h(v)$ for all $v \in F^0$ and $j \in \ZZ/w(v)\ZZ$,
and the map $h \mapsto g_h$ is bijection between graph traces on $F$ and graph traces on
$\Lambda$.
\end{lemma}
\begin{proof}
Given a graph trace $h$ on $F$, define functions $g_v : \Lambda_v^0 \to [0,\infty)$ by
$g_v(v,i) = h(v)$ for all $i$. Since each $(v,i)\Lambda_v^{e_1} = (v,i)\Lambda_v^{e_1}
(v,i+1) = \{(a_v, i)\}$ and $(v,i)\Lambda_v^{e_2} = (v,i)\Lambda_v^{e_2}(v, i+1) =
\{(b_v,i)\}$, the $g_v$ are all graph traces by Lemma~\ref{lem4.3}. Since $h$ is a graph
trace, each $h(v) = \sum_{e \in vF^1} h(s(e))$. So each
\begin{align*}
g_v(v,i)
    = \sum_{e \in vF^1} h(s(e))
    &= \sum_{e \in vE^1, j < w(s(e))/w(r(e))} g_{s(e)}\big(s(e), i + j w(r(e))\big)\\
    &= \sum_{e \in vE^1, p_e(s(e),j) = (v,i)} g_{s(e)}(s(e),j).
\end{align*}
So the $g_v$ are compatible, and there is a graph trace $g$ as claimed.

Conversely, given a graph trace $g$ on $\Lambda$, define $h : F^0 \to [0,\infty)$ by
$h(v) = g(v,0)$. Since each $(v,i)\Lambda_v^{e_1} = \{(a_v, i)\}$ and $s(a_v, i) = (v,
i+1)$, we have $g(v, i) = g(v, i+1)$ for all $i$, and so $g(v,i) = g(v,j)$ for all $v \in
E^0$ and $i,j \in \ZZ/w(v)\ZZ$. So each
\[
h(v)
    = g(v,0)
    = \sum_{\alpha \in (v,0)\Lambda^{e_3}} g(s(\alpha))
    = \sum_{e \in vE^1} w(s(e))/w(r(e)) g(s(e),0)
    = \sum_{f \in vF^1} h(s(f)).
\]
So $h$ is a graph trace, and $g = g_h$.
\end{proof}

\begin{cor}
Let $\Lambda=\Lambda_E$ be the rank-3 Bratteli diagram associated to a singly connected
weighted Bratteli diagram $E$, and take $\theta \in \RR$. Let $c^3_\theta \in Z^2(\ZZ^3,
\TT)$ be as in~\eqref{ck}, and let $c=\dstar{}c^3_\theta\in Z^2(\Lambda,\TT)$. Consider
the Bratteli diagram $F$ of Lemma~\ref{lem:g<->h}. Let $\tau$ be a semifinite
lower-semicontinuous trace on $C^*(F)$. There is a gauge-invariant semifinite
lower-semicontinuous trace $\tilde\tau$ on $C^*(\Lambda,c)$ such that
$\tilde\tau(p_{(v,i)}) = \tau(p_v)$ for all $v \in F^0$ and $i \in \ZZ/w(v)\ZZ$. The map
$\tau \mapsto \tilde\tau$ is a bijection between semifinite lower-semicontinuous traces
on $C^*(F)$ and gauge-invariant semifinite lower-semicontinuous traces on
$C^*(\Lambda,c)$. If $\theta$ is irrational then every semifinite lower-semicontinuous
trace on $C^*(\Lambda,c)$ is gauge-invariant.
\end{cor}
\begin{proof}
Lemma~\ref{lem4.5} show that each semifinite trace $\tau$ on $C^*(F)$ determines a graph
trace $h = h_\tau$ on $F$ such that $h_\tau(v) = \tau(p_v)$. Lemma~\ref{lem:g<->h} shows
that there is then a graph trace $g = g_h$ on $\Lambda$ such that $g(v,i) = h(v) =
\tau(p_v)$ for all $v \in F^0$. Now Lemma~\ref{lem4.6} and Remark~\ref{rmk:nonfaithful}
yield a gauge-invariant semifinite lower-semicontinuous trace $\tilde\tau = \tau_g$ such
that $\tilde\tau(p_{(v,i)}) = g(v,i)= h(v) = \tau(p_v)$ as claimed. We have
\[\textstyle
    C^*(F) = \overline{\bigcup_n \lsp\{s_\mu s^*_\nu : s(\mu) = s(\nu) \in E^0_n\}}.
\]
Each $\lsp\{s_\mu s^*_\nu : s(\mu) = s(\nu) \in E^0_n\} \cong \bigoplus_{v \in E^0_n}
M_{F^* v}(\CC)$ via $s_\mu s^*_\nu \mapsto \theta_{\mu,\nu}$; in particular this
isomorphism carries each $p_v$ to a minimal projection in the summand $M_{F^* v}(\CC)$.
So each trace on $C^*(F)$ is completely determined by its values on the $p_v$, and so
$\tau \mapsto \tilde\tau$ is injective.

If $\rho$ is a gauge-invariant semifinite lower-semicontinuous trace on $C^*(\Lambda, c)$
then Theorem~\ref{thm4.7} and Remark~\ref{rmk:nonfaithful} shows that $\rho = \tau_g$
where $g$ is the graph trace on $\Lambda$ such that $g(v,i) = \rho(p_{(v,i)})$. Now
Lemma~\ref{lem:g<->h} shows that $g = g_h$ and $g_h(v,i)=h(v)$ for some graph trace $h$
on $C^*(F)$, and then Theorem~\ref{thm4.7} and Remark~\ref{rmk:nonfaithful} give a
gauge-invariant semifinite lower-semicontinuous trace $\tau=\tau_h$ on $C^*(F)$ such that
$\tau(p_v)=h(v)$. Hence $\rho(p_{(v,i)})=\tilde\tau(p_{(v,i)})$. Now $\rho = \tilde\tau$
because they are both gauge-invariant traces, and so Theorem \ref{thm4.7} (and
Lemma~\ref{lem4.6}) shows that gauge-invariant traces are completely determined by their
values on vertex projections.

Suppose that $\theta$ is irrational and that $\tau$ is a semifinite lower-semicontinuous
trace on $C^*(\Lambda, c)$. For $v \in E^0$, let $c_v := c|_{\Lambda_v}$. Then $\tau$
restricts to a trace on each $C^*(\Lambda_v, c_v)$. Lemma~\ref{lem5.7} shows that each
$C^*(\Lambda_v, c_v)$ is isomorphic to $M_{w(v)}(A_\theta)$. The gauge-invariant trace
$\tau_g$ on $C^*(\Lambda,c)$ determined by the graph trace $g(v,i) = \tau(p_{(v,i)})$
restricts to a trace on each $C^*(\Lambda_v, c_v)$ such that $\tau_g(s_\mu s^*_\nu) =
\delta_{\mu,\nu} g(s(\mu)) = \delta_{\mu,\nu} \tau(p_{s(\mu)})$ for all $\mu,\nu \in
\Lambda_v$, and in particular $\|\tau_g|_{C^*(\Lambda_v, c_v)}\| = \sum_{i \in
\ZZ/w(v)\ZZ} g(v,i) = \|\tau|_{C^*(\Lambda_v, c_v)}\|$. Since there is only one trace on
$M_{w(v)}(A_\theta)$ with this norm, it follows that $\tau|_{C^*(\Lambda_v, c_v)} =
\tau_g|_{C^*(\Lambda_v, c_v)}$ and in particular
\begin{equation}\label{eq:tauonsummand}
\tau(s_\mu s^*_\nu) = \delta_{\mu,\nu} \tau(p_{s(\mu)})\quad\text{ for
$\mu,\nu \in \Lambda_v$.}
\end{equation}
Now suppose that $\alpha,\beta \in \Lambda$ and $s(\alpha) =
s(\beta)$. Write $\alpha = \eta\mu$ and $\beta = \zeta\nu$ where $\eta, \zeta\in
\Lambda^{\NN e_3}$ and $\mu,\nu \in \iota_v(\Lambda_v)$ for some $v \in E^0$. Since
$d(\eta)_2 = d(\zeta)_2 = 0$, we have $c(\eta,\mu) = 1 = c(\zeta,\nu)$. This and the
trace condition give
\[
\tau(s_\alpha s^*_\beta)
    = \tau(s_\eta s_\mu s^*_\nu s^*_\zeta)
    = \tau(s_\mu s^*_\nu s^*_\zeta s_\eta).
\]
This is zero unless $r(\zeta) = r(\eta)$, so suppose that $r(\zeta) = r(\eta) \in
\iota_w(\Lambda^0_w)$. Let $m,n \in \NN$ be the elements such that $v \in E_n^0$ and $w
\in E_m^0$. Since $s(\zeta)$ and $s(\eta)$ both belong to $\iota_v(\Lambda_v^0)$, we have
$d(\zeta) = d(\eta) = (n - m)e_3$, and then $s^*_\zeta s_\eta = \delta_{\zeta,\eta}
p_{s(\zeta)}$. So~\eqref{eq:tauonsummand} gives
\[
\tau(s_\alpha s^*_\beta)
    = \delta_{\zeta,\eta} \tau(s_\mu s^*_\nu)
    = \delta_{\zeta,\eta} \delta_{\mu,\nu} \tau(p_{s(\mu)})
    = \delta_{\alpha,\beta} \tau(p_{s(\alpha)})
    = \tau_g(s_\alpha s^*_\beta).
\]
So $\tau = \tau_g$, and is gauge invariant as claimed.
\end{proof}

\end{document}